\documentclass[a4paper, reqno]{amsart}

\usepackage{cite}
\usepackage{amscd}
\usepackage{eucal}
\usepackage{eulervm}
\usepackage{palatino}

\newtheorem{theorem}{Theorem}
\newtheorem{lemma}[theorem]{Lemma}
\newtheorem{corollary}[theorem]{Corollary}

\newtheorem{definition}[theorem]{Definition}

\theoremstyle{remark}

\newtheorem{remark}[theorem]{Remark}

\renewcommand{\theenumi}{{\rm \roman{enumi}}}
\def\Rl{{\mathbb{R}}}
\def\Cx{{\mathbb{C}}}

\def\uC{{\mathfrak{C}}}
\def\aM{{\mathcal{M}}}
\def\aN{{\mathcal{N}}}
\def\Comp{{\mathcal{K}}}
\def\sH{{\mathcal{H}}}
\def\Fred{{\mathcal{F}}}
\def\dom{{\mathrm{dom}}}
\def\Bd{{{B}}}
\def\cC{{\mathcal{L}}}

\def\Ran{{\mathrm{Ran}}}
\def\sE{{\mathcal{E}}}
\def\sL{{\mathcal{L}}}
\def\sEcross{{\sE^{\times}}}
\def\supp{{\mathrm{supp}\,}}
\def\Comp{{\mathcal{K}}}

\begin{document}

\title[Spectral flow as integral of one forms]{Spectral flow is the
  integral of one forms on \\ the Banach manifold of self adjoint
  Fredholm operators}

\author[A.~Carey]{Alan Carey}
\email{acarey@maths.anu.edu.au}
\address{Mathematical Sciences Institute, Australian National
University, Canberra ACT 0200 Australia}

\author[D.~Potapov]{Denis Potapov}
\email{d.potapov@unsw.edu.au}

\address{School of Mathematics and Statistics,
  University of New South Wales, Sydney, NSW  2052 , Australia}

\author[F.~Sukochev]{Fedor Sukochev}
\email{f.sukochev@unsw.edu.au}
\address{School of Mathematics and Statistics,
  University of New South Wales, Sydney, NSW  2052 , Australia}

\maketitle


\bibliographystyle{abbrv}

\noindent{\bf Abstract}.  
One may trace the idea that spectral flow should be given as the
 integral of a one form back to the 1974 Vancouver ICM address of I.M.
 Singer. Our main theorem gives analytic formulae for the spectral
 flow along a norm differentiable path of self-adjoint bounded
 Breuer-Fredholm operators in a semi-finite von Neumann algebra. These
 formulae have a geometric interpretation which derives from the
 proof. Namely we define a family of Banach submanifolds of all
 bounded self-adjoint Breuer-Fredholm operators and on each
 submanifold define global one forms whose integral on a norm
 differentiable path contained in the submanifold calculates the
 spectral flow along this path. We emphasise that our methods do not
 give a single globally defined one form on the self adjoint Breuer-
 Fredholms whose integral along all paths is spectral flow rather, as
 the choice of the plural `forms' in the title suggests, we need a
 family of such one forms in order to confirm Singer's idea. The
 original context for this result concerned paths of unbounded
 self-adjoint Fredholm operators. We therefore prove analogous
 formulae for spectral flow in the unbounded case as well.  The proof
 is a synthesis of key contributions by previous authors, whom we
 acknowledge in detail in the introduction, combined with an
 additional important recent advance in the differential calculus of
 functions of non-commuting operators.

\section{Introduction}
The notion of spectral flow has been a useful tool in geometry ever
since its invention by Lusztig and its application by
Atiyah-Patodi-Singer \cite{APS1,APS3}.  Until about a decade ago
spectral flow was considered primarily in topological terms as an
intersection number and there seemed to be no analytic viewpoint.  This
was despite the observation of I.M. Singer \cite{Si} that the eta
invariant is the integral of a one form and so, from the variation of
eta formula in \cite{APS1,APS3}, by implication one is led to ask
whether spectral flow is expressible as the integral of a one-form.  In
this paper we provide an answer to this question.

There has been a succession of contributions leading to our resolution
of Singer's question. We mention initial progress on an analytic
approach to spectral flow in \cite{DHK, H, Kam, G}.  Then in \cite{Ph,
  Phillips-SF2-1997} an analytic definition of spectral flow was given.
This definition applied equally to type $II$ von Neumann algebras where
operators with continuous spectrum may arise (an issue initially raised
in \cite{P1,P2}).  A key step in synthesising these developments was
taken in \cite{CarPhi1998-MR1638603, CarPhi2004-MR2053481} which exploit
an essential contribution of Getzler \cite{G} to produce spectral flow
formulae as integrals of one-forms on affine subspaces of the Banach
manifold of self adjoint Fredholms. Noncommutative geometry plays a key
role in all three of these papers in that theta and finitely summable
spectral triples are utilised. The significant new ingredient in
\cite{CarPhi1998-MR1638603, CarPhi2004-MR2053481} was the introduction
of general analytic methods which demonstrated that these analytic
spectral flow formulae apply equally to the standard type I situation
envisaged by Singer and also to spectral flow along paths of self
adjoint Breuer-Fredholm operators in a type $II_\infty$ von Neumann
algebra.  This development was partly motivated by ideas of Mathai on
$L^2$ spectral invariants for manifolds whose fundamental group has a
non-type $I$ regular representation, see \cite{M}.  For the benefit of
readers unfamiliar with the terminology above we will summarise the
relevant definitions in later Sections.  Readers interested in more
details should see the review in \cite{BCPRSW}, while readers unfamiliar
with the Breuer-Fredholm theory may consult \cite{BreuerII, CPRS3}.

A decisive further development occurred in \cite{ACS}
(see also \cite{ACDS}). The functional
calculus methods of \cite{CarPhi1998-MR1638603, CarPhi2004-MR2053481}
simply do not generalise sufficiently to answer Singer's question. A
more sophisticated functional calculus is needed and this was provided
partly in \cite{ACS} where it is explained how double operator integrals
(DOI) give a differential calculus for functions of operators.  It is
the key technical tool we exploit in the present paper and in order to
make the discussion more self contained we develop, in Section 5, the
relevant parts of this DOI technique.  A second innovation, which
occurred in \cite{Wahl}, was a new way to handle paths of unbounded self
adjoint Fredholm operators.  In \cite{Wahl} a spectral flow formula for
paths that lie in an affine space of relatively bounded perturbations of
a fixed unbounded self adjoint Fredholm operator was proved.  This
inspired the present work whose principal aim is to give a very general
answer to Singer's question in the case of the Banach manifold of
bounded self adjoint Fredholm operators and then to deduce a
generalisation of the unbounded results of \cite{Wahl} from our bounded
formula.  We emphasise that the methods are sufficiently strong to
answer Singer's question in a general semifinite von Neumann algebra.

To illustrate our ideas we now summarise a special case of our results.
Suppose~$\aM$ is a semi-finite von Neumann algebra with a normal
semifinite faithful (n.s.f.) trace $\tau$ which will be fixed
throughout.  We take the $\tau$-Calkin algebra to be the quotient of
$\mathcal M$ by the norm closed ideal generated by the $\tau-$finite
projections.  An operator is $\tau$-Fredholm (and hence Breuer-Fredholm)
if it is invertible in the $\tau-$Calkin algebra. Suppose that $t
\mapsto F_t \in \aM$ is a piecewise $C^1$-path of self adjoint
$\tau$-Fredholm operators such that~$\left\| F_t \right\| \leq 1$ and
the spectrum of the image of~$F_t$ in the $\tau-$Calkin algebra
is~$\left\{ \pm 1 \right\}$.
If the endpoints of this path,~$F_0$ and~$F_1$, are unitarily equivalent, then the spectral flow,
$sf(F_t)$, may be computed by the following analytic formula
\begin{equation}
  \label{BasicSFformula}
  sf (F_t) = \int_0^1
  \tau \left( \dot F_t \, h(F_t) \right)\, dt, 
\end{equation}
where~$h$ is a positive sufficiently smooth function
on $[-1,1]$.
The choice of $h$ is dictated by the requirement
that the RHS of~(\ref{BasicSFformula}) 
is well defined, namely that
\begin{equation}
  \label{BasicAssumption}
  \int_0^1 \left\| \dot F_t \, h(F_t) \right\|_1 \, dt < + \infty,
\end{equation}
where~$\left\| \cdot \right\|_1$ is the trace norm on~$\aM$.

In some special cases, formula~(\ref{BasicSFformula}) has been proved by
different methods in~\cite{Wahl, ACS, CarPhi1998-MR1638603,
  CarPhi2004-MR2053481} under significant additional restrictions on the
path~$\{F_t\}$.  A feature of the methods we employ in this paper
is that  we are able to remove the 
assumption of~\cite{ACS, CarPhi2004-MR2053481, CarPhi1998-MR1638603},
that Fredholm paths~$\{F_t\}$ must lie in the affine space of 
$\tau$-compact
perturbations of a fixed Fredholm operator $F_0$.  This affine space is
contractible so that the spectral flow of any loop in the space is zero
and hence these affine space formulae do not directly reveal the rich topology of the
space of Breuer-Fredholm operators. Every such affine space lies entirely within one
of the submanifolds described in the abstract and one may recover the `global' formula of the affine space studied in
\cite{ACS, CarPhi2004-MR2053481, CarPhi1998-MR1638603}
by an approximation argument (although we give details only in 
the unbounded case).
We emphasise that the
consequences of the spectral flow formulae
in ~\cite{CarPhi2004-MR2053481, CarPhi1998-MR1638603}  for affine spaces are quite
profound: they imply for example the local index formula in
noncommutative geometry in semifinite spectral triples \cite{CPRS2,CM}.  

In
the present paper we shall show that a modification of the approach
of~\cite{ACS} allows us to prove~(\ref{BasicSFformula}) under only the
requirement~(\ref{BasicAssumption}).
The formula~(\ref{BasicSFformula}) is a special case of a much more
general formula which we prove in Section 3 of this paper.  In Section 3
we give expressions for spectral flow along norm differentiable paths
in the Banach manifold of bounded self adjoint $\tau$-Fredholms in a
semifinite von Neumann algebra. The assumptions under which these formulae
hold are the minimal ones: there
are no unnecessary side conditions.  As
Singer's question was originally phrased in the case of spectral flow
along paths of unbounded self adjoint operators, we deduce, in Section
4, unbounded formulae from our bounded one. Namely we prove that, for a
pair $D_0$, $D_1$ of unbounded self adjoint $\tau$-Fredholms, spectral
flow along any path joining them that is smooth in the graph norm of
$D_0$ is the integral of a one form defined on the affine space of
$D_0$-graph norm bounded self adjoint perturbations of $D_0$.  This is a
strengthening of all previous results (in particular, those
in~\cite{ACS, Wahl, CarPhi1998-MR1638603, CarPhi2004-MR2053481}).

We now summarise the geometric meaning of (\ref{BasicSFformula}).  In
Section 3 we give more details.  The space $\Fred_*^{\pm 1}$ of all
self adjoint $\tau$-Fredholms with norm less than or equal to one and
with essential spectrum $\pm 1$ plays a special role in the theory as we
will see later.  (In the Appendix we show that the well known lemma of
\cite{AS} that the space $\Fred_*^{\pm 1}$ is a deformation retract of
the space of self adjoint Fredholms with both positive and negative
essential spectrum still holds in a semifinite von Neumann algebra.) 
As $\Fred_*^{\pm 1}$ itself does not appear to be a manifold we need to take care in interpreting our construction. 
We start with an auxiliary bigger class namely all self-adjoint $\tau$-Fredholm operators
with no essential spectrum in the interval~$[-\delta, 
\delta]$ for some~$\delta > 0$. As will be seen later, this
class, denoted  $\Fred_\delta$, is an open subset of the self-adjoint part of the algebra~$\aM$
and therefore, clearly, is a Banach manifold.  
Any norm continuous path $t\to F_t,$ $t\in [0,1]$
in the self adjoint $\tau$-Fredholm operators
lies in a $\Fred_\delta$ for some $\delta>0$. Then the integrand of
(\ref{BasicSFformula}) is a one form in the following sense. The
integrand comes from the functional $\theta_F$ defined on the tangent
space to the manifold of self adjoint $\tau$-Fredholms at $F\in\Fred_\delta$ by
$\theta_F(X)=\tau (Xh(F))$ for a suitably chosen $C^1$ function~$h$
with support in $[-\delta,\delta]$. We will see that this functional gives an exact one form
on sufficiently small convex neighbourhoods of $F$. (Note that this
geometric viewpoint can be traced back to \cite{G}).  Thus
(\ref{BasicSFformula}) is to be interpreted as the integral of a one form on a path in 
$\Fred_\delta$ for $\delta<1$. 
Note however that, in our approach, there is no global one form
on the space of all $\tau$-Fredholm operators that calculates
spectral flow.  It is necessary to vary $\delta$ and hence the function $h$
depending on the path in question. Our
general formula (which we do not  state in this introduction as it requires much more
notation than (\ref{BasicSFformula}))
applies when the endpoints are not unitarily equivalent.

For suitable paths, and hence functions $h_\delta$, we may take $\delta\to1$ and in this way we recover the unbounded affine space formulae of \cite{CarPhi2004-MR2053481, CarPhi1998-MR1638603}. However the geometric interpretation
that the resulting formulae are integrals of  one forms on the affine space has to be reproved ab initio (and we do not do this here cf.  \cite{CarPhi2004-MR2053481, CarPhi1998-MR1638603}). 
Also, for the unbounded case in Section 4, we first remark that an unbounded
self adjoint operator $D$ is $\tau$-Fredholm if the operator
$F_D=D(1+D^2)^{-1/2}$ is a $\tau$-Fredholm operator in $\mathcal M$.
The issue of the differentiability of the
map $D\mapsto F_D$ as a function on the unbounded $\tau$-Fredholms has
proved in the past to be the principal obstacle to proving spectral flow
formulae for the unbounded case using formulae for the bounded case (see
for example the discussion in \cite{Wahl}). One of the novelties of our
approach in this paper is a very satisfactory resolution of this
differentiability question described in Section 6.  This is used in
Section 4 to obtain a straightforward proof of the unbounded formula.
As in \cite{G} the motivation for our approach comes from questions in
noncommutative geometry.  In the next Section we will explain one
relationship of our results to the latter formalism.  This enables one
to understand spectral flow as a pairing of K-homology with K-theory.

The remainder of the paper is organised as follows. We prove the most general formula for 
spectral flow along paths of bounded self adjoint $\tau$-Fredholm operators in Section 3. In Section 4 we deduce 
from the bounded formula a corresponding formula for paths of unbounded self adjoint $\tau$-Fredholms. We present the proofs in as direct a fashion as possible deferring technical issues
on double operator integrals to Section 5 and background on graph norm bounded paths of unbounded operators to Section 6. We present a reasonably detailed discussion in Sections 5 and 6 to make this paper more self contained and independent of previous papers on these two topics.

{\bf Acknowledgements}. This research was supported by the Australian Research Council and the Hausdorff Institute for Mathematics. We also thank the referee and
Matthias Lesch for valuable comments on an earlier draft.

\section{Perturbations of spectral triples}

To explain how the calculation of spectral flow presented in the
following Sections fits into the overall picture in noncommutative
geometry \cite{C} we describe some preliminary results in this Section.

A semifinite spectral triple consists of an unbounded self adjoint
operator $D$ on a Hilbert space $\mathcal H$, a 
unital $*$-subalgebra $\mathcal
A$ of a semifinite von Neumann algebra $\mathcal M$ (with faithful
normal semifinite trace $\tau$) acting on $\mathcal H$ such that the
commutator $[D,a]$ extends to a bounded linear operator on $\mathcal H$
for all $a\in \mathcal A$ and with $D$ having $\tau$-compact resolvent
in $\mathcal M$.

In previous work spectral flow between operators in the affine space of
bounded self adjoint perturbations of $D$ was studied in the context of spectral triples
and a formula for
spectral flow proved that provides a first step in the resolution of
Singer's question.  In \cite{CarPhi1998-MR1638603} it is observed that
if $A$ is bounded then $$
F_D-F_{D+A}:=D(1+D^2)^{-1/2}-(D+A)(1+(D+A)^2)^{-1/2} $$ is 
$\tau$-compact. This observation is crucial to the method of proof of the spectral
flow formulae in \cite{CarPhi1998-MR1638603,CarPhi2004-MR2053481},
namely, one deduces the unbounded formula from a formula for spectral
flow in the affine space of $\tau$-compact perturbations of a fixed bounded
$\tau$-Fredholm operator $F$.  In \cite{CarPhi1998-MR1638603} it was
observed that if $A$ is an unbounded self adjoint operator affiliated to
$\mathcal M$ that is bounded in the graph norm of a fixed unbounded self
adjoint operator $D$ (that is, $\dom ({D}) \subseteq \dom (A) $ and
$||Av||\leq c(||v||+ ||Dv||)$ for some $c>0$ and all $v\in$ Dom$(A)$)
affiliated to $\mathcal M$ then $F_D-F_{D+A}$ is bounded but is not
$\tau$-compact in $\mathcal M$.  Thus to prove a spectral flow formula for
spectral flow between $D$ and a graph norm bounded perturbation using
the strategy of \cite{CarPhi1998-MR1638603} requires us to prove a
formula for spectral flow for general paths of bounded self adjoint Fredholm
operators.

The noncommutative geometry framework in the bounded case is that of
semifinite pre-Fredholm modules (a special case of Kasparov modules \cite{C}).
For our purposes we will only need the following definition.  With
$\mathcal A$ and $\mathcal M$ as above a semifinite pre-Fredholm module
is given by a self adjoint operator $F$ in $\mathcal M$ such that
$1-F^2$ is $\tau$-compact and $[F,a]$ is $\tau$-compact for all $a\in
\mathcal A$.  (If $F^2=1$ then the prefix `pre' is dropped.)

\begin{lemma}
  \begin{enumerate}
  \item Any semifinite spectral triple for $\mathcal A$ defines a
    semifinite pre-Fredholm module where we choose $F$ to be
    $F_D=D(1+D^2)^{-1/2}$. \label{PSTmodule}

  \item If $A$ is a self adjoint unbounded operator such that the $D$-graph
    norm of~$A$ is less than~$1$ and $[A,a]$ is bounded for all $a\in
    \mathcal A$ then $D+A$ also defines a spectral triple for $\mathcal
    A$. The semifinite pre-Fredholm module for $\mathcal A$ given by
    $F_{D+A}$ is homotopic to that given by $F_D$ with the homotopy
    defined by the path $\{F_{D+tA}$, $t\in [0,1]\}$.
    \label{PSThomotopy}
  \end{enumerate}
\end{lemma}

\begin{proof}
  (\ref{PSTmodule})~It is sufficient to observe that, since~$D$ has a
  $\tau$-compact resolvent, the operator~$1 - F^2_D$ is $\tau$-compact.
  Observe also that if~$[D, a]$ is bounded, then~$[F_D, a]$ is
  $\tau$-compact (see~\cite[Theorem~11]{PoSuNGapps}).

  (\ref{PSThomotopy})~To see that~$D+A$ defines a spectral triple
  for~$\mathcal A$, we have to note that the operator~$(1+(D+A)^2)^{-1}$
  is $\tau$-compact (see~\cite[Appendix~B,
  Lemma~7]{CarPhi1998-MR1638603}).  In order to define a homotopy, the
  path~$\left\{ F_{D+tA},\ t \in [0, 1] \right\}$, should be continuous
  with respect to the operator norm.  This follows from the fact (see
  Section~\ref{sec:paths-self-adjoint} and the
  identity~(\ref{BasicIdea}) in particular) that there is a uniformly
  bounded family of continuous linear operators~$\left\{ {T_t}
  \right\}_{0 \leq t \leq 1}$ on~$\aM$ such that $$ F_{D + tA} - F_D =
  t\, T_t (A (1 + D^2)^{-\frac 12}). $$ The lemma is proved.
\end{proof}

Phillips definition of spectral flow \cite{Ph, Phillips-SF2-1997}, which
is extended and explained in some detail in \cite{BCPRSW}, depends on a
simple observation.  Let $\chi$ be the characteristic function of
$[0,\infty)$. Let $\mathcal N$ be a semifinite von Neumann algebra
 with semifinite,
faithful, normal trace, $\tau$.  Let $\{F_t\}$  be any norm
continuous path in the bounded self adjoint $\tau$-Fredholms
in $\mathcal N$ (indexed by some interval $
[a,b]$).   If we let $\pi$ be the projection onto the Calkin algebra
then one may show that
 $\pi\left(\chi (F_t)\right)=\chi\left(\pi (F_t)\right)$. As 
the spectra of $
\pi (F_t)$ are
bounded away from 0, this latter path is continuous.  By compactness
we can choose a partition $a=t_0<t_1<\cdots <t_k=b$ so that for each
$i=1,2,\cdots ,k$
$$||\pi\left(\chi (F_t)\right)-\pi\left(\chi (F_s)\right)||
<\frac{1}{2}\quad\hbox{for all }t,s\hbox{ in }[t_{i-1},
t_i].$$
Letting $P_i=\chi (F_{t_i})$ for $i=0,1,\cdots ,k$ then by the previous inequality
(see \cite{BCPRSW}) $P_{i-1}P_i: P_i H\to P_{i-1}H$ is Fredholm.
Then we define the {\it spectral 
flow of
the path $\{F_t\}$} to be the number:
$$sf\left(\{F_t\}\right)=\sum_{i=1}^k \mbox{ind}(P_{i-1}P_i)$$
which is independent of the choice of partition of $[a,b]$. 
This analytic point of view recovers
the intersection number approach to spectral flow when the operators in question have
discrete spectrum.

The spectral flow for a $D$-graph norm continuous path $\{D_t:\ t\in
[0,1]\}$ of unbounded self adjoint $\tau$-Fredholms joining $D=D_0$ to
$D_1$ affiliated to $\mathcal M$ is defined as the spectral flow along the corresponding path
$F_{D_t}$ of bounded $\tau$-Fredholms.  
When $u\in \mathcal A$ the spectral flow along the path 
$$D_t:=(1-t)D + t uDu^*= D+tu[D,u^*]$$ defines a pairing between
the K-homology class defined by the semifinite spectral triple $({\mathcal H}, D, {\mathcal A})$
and the
class of $u$ in $K_1({\mathcal A})$. The preceding lemma gives a
condition on the perturbation $A$ under which the spectral triple
defined by $D+A$ gives the same pairing with $K_1({\mathcal A})$ as does
the spectral triple defined by $D$.

Notice that for an unbounded self adjoint operator $D$ with
$(1+D^2)^{-1}$ being $\tau$-compact the map $D\mapsto F_D$ has range in
the space of self adjoint bounded $\tau$-Fredholms of norm less than or
equal to one and such that the essential spectrum is contained in $\pm
1$.  That is $1-F_D^2=(1+D^2)^{-1}$ is $\tau$-compact, explaining in
part the distinguished role played by this retract of the manifold of
all bounded $\tau$-Fredholm operators in the subsequent exposition.
 
\section{Spectral flow formula, bounded case.}
\label{sec:spectr-flow-form}

Let~$\aM$ be a von Neumann algebra and let~$\tau$ be a n.s.f.~trace
on~$\aM$.  $\left\| \cdot \right\|$ stands for the operator norm
on~$\aM$.  Let~$L^1(\aM)$ be the predual of~$\aM$ equipped with the
trace norm~$\left\| \cdot \right\|_1$.  Recall that an operator~$F \in
\aM$ is called $\tau$-Fredholm if and only if
\begin{enumerate}
\item the projections~$N_F$ and~$N_{F^*}$ are $\tau$-finite;
\item there is a $\tau$-finite projection~$p \in \aM$, such that~$\Ran(1-p)
  \subseteq \Ran(F)$.
\end{enumerate}
Here~$N_F$ is the projection onto the~$Ker(F)$ and~$\Ran(F)$ is the
range of the operator~$F$.

Let~$\Comp$ be the two-sided ideal of all $\tau$-compact operators
of~$\aM$.  The quotient space~$\aM / \Comp$ is a $C^*$-algebra.
Let~$\pi$ be the homomorphism $$ \pi : \aM \mapsto \aM / \Comp. $$
Recall the following characterization of the $\tau$-Fredholm operators
due to M.~Breuer (see~\cite[Theorem~1]{BreuerII}): {\it An operator~$F$
  is $\tau$-Fredholm if and only if the image~$\pi(F)$ is invertible.}
We set~$\delta_F = \left\| \pi(F)^{-1} \right\|^{-1}$.  Note that the
mapping~$F \mapsto \delta_F$ is continuous on the 
Banach manifold of $\tau$-Fredholm operators.

We shall denote the set of all self adjoint $\tau$-Fredholm operators~$F
\in \aM$ by~$\Fred_*$.  We also shall denote the subset of~$\Fred_*$
with~$\left\| F \right\| \leq 1$ and~$\delta_F = 1$ by~$\Fred_*^{\pm
  1}$.

The characterization of M.~Breuer implies that if~$F$ is a self adjoint
$\tau$-Fredholm operator, then, for every~$0 \leq \delta < \delta_F$,
the spectral projection~$E^{F}(-\delta, \delta)$ is $\tau$-finite i.e.,
$$ \tau \left( E^F(-\delta, \delta) \right) < + \infty,\ \ 0 \leq \delta
< \delta_F. $$ Indeed, fix~$0 \leq \delta < \delta_F$.  Consider the
operator~$F_\delta = F - F \, E^F(-\delta, \delta)$.  We then have that
$$ \left\| \pi(F) - \pi (F_\delta) \right\| \leq \left\| F - F_\delta
\right\| \leq \delta < \delta_F = \left\| \pi(F)^{-1} \right\|^{-1}. $$
Consequently, the operator~$\pi(F_\delta)$ is invertible and
therefore~$F_\delta$ is $\tau$-Fredholm.  This means that there is a
$\tau$-finite projection~$p$ such that~$1 - p \subseteq \Ran(F_\delta)$.
The latter implies that~$E^F(-\delta, \delta) \subseteq p$ which means
that the projection~$E^{F}(-\delta, \delta)$ is $\tau$-finite.
Furthermore,

\begin{lemma}
  \label{FredholmProperty}
  \begin{enumerate}
  \item For every~$F \in \Fred_*$ and every bounded Borel function~$g$
    supported on the interval~$[-\delta_{F}, \delta_{F}]$ such
    that~$\lim_{x \rightarrow \pm \delta_F} g(x) = g(\pm \delta_F) = 0$,
    the operator~$g(F)$ is $\tau$-compact.  In particular, if~$F \in
    \Fred_*^{\pm 1}$, then~$1 - F^2$ and~$F - B$ are $\tau$-compact,
    where~$B = 2 \chi_{[0, + \infty)}(F) - 1$. \label{TauCompact}
  \item For every~$F_0 \in \Fred_*$ and every~$0 < \delta < \delta_{F_0}$,
    there is a neighbourhood~$N$ of\/~$F_0$ such that the mapping~$F
    \mapsto E^{F}(-\delta, \delta)$ is trace norm bounded on the
    self adjoint part of~$N$. \label{FredholmPropertyBdd}
  \end{enumerate}
\end{lemma}

\begin{proof}
  (\ref{TauCompact})~To see that the operator~$g(F)$ is $\tau$-compact,
  it is sufficient to show that, for every~$\epsilon > 0$, there is a
  $\tau$-compact operator~$K_\epsilon$ such that~$\left\| g(F) -
    K_\epsilon \right\| < \epsilon$.

  Fix~$\epsilon > 0$.  Let~$x_\epsilon$ be the point~$0 < x_\epsilon <
  \delta_F$ such that $$ \left| g(x) \right| < \epsilon,\ \ \text{for
    every}\ x_\epsilon \leq \left| x \right| \leq \delta_F. $$ We
  set~$K_\epsilon = g(F)\, \chi_{[-x_\epsilon, x_\epsilon]}(F)$.  Since
  the function~$g$ is bounded and the projection~$~\chi_{[-x_\epsilon,
    x_\epsilon]}(F)$ is $\tau$-finite, the operator~$K_\epsilon$ is
  $\tau$-compact.  On the other hand, by 
  the choice of~$x_\epsilon$, we see
  that $$ \left\| g(F) - K_\epsilon \right\| \leq \sup_{x_\epsilon \leq
    \left| x \right| \leq \delta_F} \left| g(x) \right| < \epsilon. $$
  The proof is finished.

  (\ref{FredholmPropertyBdd})~Assume first that~$F_0 \in \Fred_*^{\pm
    1}$.  In this special case, the proof employs the argument
  of~\cite[Lemma~1.26.(ii)]{ACS}.  Let~$F = F_0 + A$ where~$A$ is
  self adjoint and~$\left\| A \right\| \leq \frac 16 \, (1 - \delta^2)$.
  Clearly, $$ 1 - F^2 = 1 - F_0^2 + B $$ where~$B = F_0 A + A F_0 - A^2$
  and~$\left\| B \right\| \leq \frac 12 \, (1 - \delta^2)$.
  Let~$\mu_t(X)$ be a decreasing rearrangement of the operator~$X \in
  \aM$ (see~\cite{FK1986}).  Note that $$ \chi_{(-\delta, \delta)} (x)
  \leq \frac {1 - x^2}{1 - \delta^2},\ \ \left| x \right| \leq 1.
  $$ This observation together with~\cite[Lemma~2.5]{FK1986} implies
  that
  \begin{multline*}
    \mu_t\left( E^F(-\delta, \delta) \right) \leq \frac 1{1 - \delta^2}
    \, \mu_t \left( 1 - F^2 \right) \\ \leq \frac 1{1 - \delta^2} \left[
      \mu_{\frac t2} \left( 1 - F_0^2 \right) + \mu_{\frac t2} (B)
    \right] \\ \leq \frac 1 {1 - \delta^2}\, \mu_{\frac t2} \left( 1 -
      F_0^2 \right) + \frac 12 .
  \end{multline*}
  Since the operator~$1 - F_0^2$ is $\tau$-compact, the function~$t
  \mapsto \mu_t (1 - F_0^2)$ is decreasing to~$0$ at~$+ \infty$.  Thus,
  we see that the functions~$t \mapsto \mu_t \left( E^F(-\delta, \delta)
  \right)$ are uniformly majorized across $$ F \in N = \left\{ F_0 + A,\
    \ \left\| A \right\| \leq \frac 16\, (1 - \delta^2) \right\} $$ by a
  single decreasing function with value~$\frac 12$ at~$+ \infty$.  On
  the other hand, we know that $$ \mu_{t} \left( E^F(- \delta, \delta)
  \right) = \chi_{\left[ 0, \tau\left( E^F(-\delta, \delta) \right)
    \right]}. $$ Consequently, the value $$ \tau \left( E^F(-\delta,
    \delta) \right) $$ is uniformly bounded across~$N$.
  
  Let now~$F_0 \in \Fred_*$.  Let~$\theta$ be a $C^2$-function such that
  (i)~$\theta'$ is nonnegative and supported on the
  interval~$[-\delta_{F_0}, \delta_{F_0}]$; (ii)~$\theta(\pm \infty) =
  \pm 1$.  Clearly,~$\theta(F_0) \in \Fred_*^{\pm 1}$.  Moreover, the
  mapping~$F \mapsto \theta(F)$ is operator norm continuous (see
  Remark~\ref{OperNormCont}).  Consequently, the claim of the lemma for
  the general~$F_0 \in \Fred_*$ holds with the preimage~$\theta^{-1}(N)$
  of the ball~$N$ constructed with respect to the operator~$\theta(F_0)
  \in \Fred_*^{\pm 1}$ above.
\end{proof}

\def\Fred{{\mathcal F}}

If~$t \mapsto F_t \in \Fred_*$ is a continuous path of self adjoint
$\tau$-Fredholm operators, then $sf(F_t)$ stands for the spectral flow
as defined in~\cite{BCPRSW, Phillips-SF2-1997}.  We shall prove the
following analytic formula for spectral flow, which extends that
of~\cite[Theorem~6.4]{Wahl} and~\cite[Theorem~3.18]{ACS}.

\begin{theorem}
  \label{thm:SFformulaI}
  Let~$F_t: [0, 1] \mapsto \Fred_*$ be a piecewise $C^1$-path of self
  adjoint $\tau$-Fredholm operators.  If\/~$h$ is a positive
  $C^2$-function supported on~$[-\delta, \delta]$, where~$\delta =
  \min_{0 \leq t \leq 1} \delta_{F_t}$, such that
  \begin{enumerate}
  \item $\int_{-\delta}^{+\delta} h(x)\, dx = 1$;  \label{SFformulaIi}
  \item $ \int_0^1 \left\| \dot F_t \, h(F_t) \right\|_1 \, dt < +
    \infty$; \label{SFformulaIii}
  \item $H(F_1) - H(F_0) + \frac 12 B_0 - \frac 12 B_1 \in L^1(\aM)$,
    where~$H(x)$ is an antiderivative of~$h(x)$ such that~$H(\pm \delta)
    = \pm \frac 12$ and~$B_j$ is the phase of\/~$F_j$, i.e., $B_j =
    2\chi_{[0, +\infty)}(F_j) - 1$, $j=0,1$; \label{SFformulaIiii}
  \end{enumerate}
  then
  \begin{equation}
    \label{SFformulaI}
    sf(F_t) = \int_0^1 \tau \left( \dot F_t \, h(F_t)
    \right)\, dt + \tau \left( H(F_1) - H(F_0) + \frac 12 B_0 - \frac 12
      B_1 \right).
  \end{equation}
\end{theorem}

\begin{remark}
(i) Observe that every positive $C^1$-function~$h$, which is supported
on a proper subinterval of~$[-\delta, \delta]$ (where~$\delta$ is
defined in Theorem~\ref{thm:SFformulaI} and such that $$
\int_{-\delta}^{\delta} h(x)\, dx = 1, $$ satisfies the conditions of
Theorem~3.  Indeed, (\ref{SFformulaIi}) is trivial;
(\ref{SFformulaIii}) follows from
Lemma~\ref{FredholmProperty}.(\ref{FredholmPropertyBdd}) which implies
that the function~$t \in [0, 1] \mapsto \left\| h(F_t) \right\|_1$ is
bounded; and (\ref{SFformulaIiii}) follows from
Lemma~\ref{FredholmProperty}.(\ref{FredholmPropertyBdd}) again and the
observation that the function $$ \chi_{(0, +\infty)}(x) - \frac 12 -
H(x) $$ is bounded and supported on a proper subinterval of~$[-\delta,
\delta]$.

(ii) In previous papers the case where we work in a subset of the 
$\tau$-Fredholms consisting of operators $F$ satisfying the condition $(1-F^2)^{n/2}$ is trace class \cite{CarPhi1998-MR1638603}
or $e^{-|1-F^2|^{-1}}$ is trace class \cite{CarPhi2004-MR2053481} (the $n$-summable or theta summable cases respectively) were studied.
Thus in the setting of Theorem \ref{thm:SFformulaI}  we 
would choose $h$ to be given on $[-1,1]$ by either $h(x)= (1-x^2)^{n/2}$ or
$h(x)=e^{- |1-x^2|^{-1}}$ and to be zero on the complement of $[-1,1]$.
Notice that these two functions do not satisfy the assumptions of the theorem
if we allow operators with essential spectrum $\pm 1$. This minor difficulty is handled by an approximation argument which we describe in the proofs below.
\end{remark}

Let~$t\in [0, 1] \mapsto F_t$ be a loop of self adjoint $\tau$-Fredholm
operators, $F_0 = F_1$.  It is shown
in~\cite[Remark~2.4]{Phillips-SF2-1997} that if the loop~$F_t$ lies
within sufficiently small neighbourhood in~$\aM$, then the spectral flow
along this loop is~$0$.  One of the steps in proving the analytic formula
of Theorem~\ref{thm:SFformulaI} is to show that the
integral~(\ref{SFformulaI}) is also~$0$ for such loops.  This is
precisely the part where the proof of the spectral formula
in~\cite{Wahl} exploits the assumption that the mapping~$t \mapsto 1 -
F_t^2$ is~$C^1$ with respect to the trace norm.  We shall first see
that, with some modification, the proof of~\cite[Proposition~3.5]{ACS}
allows us to avoid the latter restriction (see Theorem~\ref{SmallLoopsSF}
below).  This modification is based on results from~\cite{PS-DiffP,
  PSW-DOI, PoSu}.  

Let~$F_0 \in \aM$ be a $\tau$-Fredholm operator and let~$N(F_0)$ be the
neighbourhood given by $$ N(F_0) = \left\{ F \in \aM,\ \ \left\|
    \pi\left( F - F_0 \right) \right\| < \delta_{F_0} \right\}. $$
Clearly, applying M.~Breuer's result, every~$F \in N(F_0)$ is
$\tau$-Fredholm.  Observe also that~$N(F_0)$ is convex.

\begin{theorem}
  \label{SmallLoopsSF}
  Let~$t \in [0, 1] \mapsto F_t \in \Fred_*$ be a piecewise $C^1$-loop
  ($F_0 = F_1$) of $\tau$-Fredholm operators such that~$F_t \in N$, $t
  \in [0, 1]$, where~$N$ is an open convex subset of\/~$\aM$ such that
  the norm closure~$\bar N$ is a subset of~$N(F_0)$.  If\/~$h$ is a
  positive $C^2$-function such that
  \begin{enumerate}
  \item $\supp h \subseteq [-\delta, \delta]$ where~$\delta = \min_{F
      \in N} \delta_F>0$;
  \item $\int_0^1 \left\| \dot F_t \, h(F_t) \right\|_1 \, dt < +
    \infty$,
  \end{enumerate}
  then $$ \int_0^1 \tau \left( \dot F_t \, h(F_t) \right)\, dt = 0. $$
\end{theorem}

\begin{proof}
  Observe that, since the space~$L^1([0, 1], L^1(\aM))$ (= the space of
  all Bochner integrable $L^1(\aM)$-valued functions on~$[0, 1]$) is
  separable, for every function~$h$ such that~$\supp h \subseteq
  [-\delta, \delta]$ there is a sequence\footnote{To this end, it is
    sufficient to construct a sequence of functions~$h_n$ such that the
    difference~$h-h_n$ is uniformly bounded and the support of the
    difference~$h-h_n$ vanishes as~$n \rightarrow \infty$ and then refer
    to, e.g.~\cite[Proposition~2.1]{CS1994}.  Observe also that one can
    also achieve that $$ \int_{\Rl} h_n (x)\, dx = 1,\ \ \forall n \geq
    1. $$\label{ToProperSubinterval}} of positive
  $C^2$-functions~$\left( h_n \right)_{n = 1}^\infty$ such that $$ \supp
  h_n \subseteq (-\delta, \delta) \ \ \text{and}\ \ \lim_{n \rightarrow
    \infty} \int_0^1 \left\| \dot F_t h_n(F_t) - \dot F_t h(F_t)
  \right\|_1 \, dt = 0. $$ Consequently, $$ \int_0^1 \tau \left( \dot
    F_t \, h_n (F_t) \right)\, dt = 0,\ \forall n \geq 1\ \
  \Longrightarrow\ \ \int_0^1 \tau \left( \dot F_t\, h(F_t) \right)\, dt
  = 0. $$ Thus, without loss of generality, we may assume
  that~$\supp h \subseteq (-\delta, \delta)$.

  Let~$\supp h \subseteq (-\delta, \delta)$ and let~$g$ be a positive
  function such that~$\supp g \subseteq (-\delta, \delta)$, $g^{\frac
    12} \in C^2$ and~$g(x) = 1$ for every~$x$ in some neighbourhood
  of~$\supp h$.  Let us show that the mapping $t \in [0, 1] \mapsto
  g(F_t)$ is a continuous function with respect to the trace norm.
  Indeed, by the representation $$ g(F_t) - g(F_s) = g^{\frac 12} (F_t)
  \left( g^{\frac 12} (F_t) - g^{\frac 12} (F_s) \right) + \left(
    g^{\frac 12}(F_t) - g^{\frac 12} (F_s) \right)\, g^{\frac 12}(F_s),
  $$ is it clear that this mapping is continuous in the trace norm
  provided the function $t \mapsto g^{\frac 12} (F_t)$ is continuous in
  the operator norm and bounded in the trace norm.  For the former, note
  that~$t \mapsto F_t$ is operator norm continuous and~$g^{\frac 12} \in
  C^2$ and therefore the path~$t \mapsto g^{\frac 12}(F_t)$ is also
  operator norm continuous (see Remark~\ref{OperNormCont}).  For the
  latter, observe that (i)~$0 \leq g(x) \leq \chi(x)$, for some
  indicator function~$\chi$ of a proper subinterval of~$(- \delta,
  \delta)$; (ii)~consequently, according to
  Lemma~\ref{FredholmProperty}.(\ref{FredholmPropertyBdd}), the
  function~$t \mapsto g^{\frac 12} (F_t)$ is bounded with respect to the
  trace norm in a small neighbourhood of every point~$t \in [0, 1]$;
  (iii)~finally, due to compactness, this function is also globally
  bounded.  Observe also that, since~$h$ is $C^2$, the mapping~$t
  \mapsto h(F_t)$ is operator norm continuous (see
  Remark~\ref{OperNormCont}).  Furthermore, since~$h(F_t) = h(F_t)
  g(F_t)$, the mapping~$t \mapsto h(F_t)$ is also continuous with
  respect to the trace norm.  We shall single out the argument presented
  above as the following lemma.

  \begin{lemma}
    \label{FredholmContArgument}
    For every~$F_0 \in \Fred_*$ and every $C^2$-function~$h$ supported
    on a proper subinterval of~$[-\delta_{F_0}, \delta_{F_0}]$, there is
    a neighbourhood~$N$ of~$F_0$ such that the mapping~$F \in N \mapsto
    h(F)$ is trace norm continuous on the self adjoint part of~$N$.
\end{lemma}

  From this point on, the proof of Theorem~\ref{SmallLoopsSF} follows 
  that of~\cite[Proposition~3.5]{ACS}.  We shall show that there is a
  function~$\theta : N \mapsto \Cx$ such that   \begin{equation}
    \label{LastObjective}
    d\theta_F(X) =
    \tau(X h(F)),\ \ X \in \aM. 
  \end{equation}
  In other words, we shall show that the one form~$\tau \left( X h(F)
  \right)$ is exact.  This will finish the proof.

  Fix the element~$F \in N(F_0)$.  For the rest of the proof, $r \in [0,
  1] \mapsto F_r \in N$ stands for the straight line path
  connecting~$F_0$ and~$F$ (i.e., $F_r = (1 - r) F_0 + r F$).  We
  introduce the function~$\theta: N \mapsto \Cx$ as follows $$ \theta(F)
  = \int_0^1 \tau \left( \dot F_r \, h(F_r) \right)\, dr. $$ Let~$d
  \theta_F(X)$, $X \in \aM$ be the differential form of~$\theta$, i.e.,
  $$ d \theta_F(X) := \lim_{s \rightarrow 0} \frac 1s \left( \theta(F +
    sX) - \theta(F) \right).  $$ Now we prove~(\ref{LastObjective}).
  Fix~$X \in \aM$.  Following the definition of the function~$\theta$, a
  simple computation yields
  \begin{align}
    \label{LastObjectiveI}
    \frac 1s \left( \theta(F + sX) - \theta(F) \right) = &\, \int_0^1
    \tau \left( X\, h(F_r + sr X) \right)\, dr \cr + &\, \int_0^1 \tau
    \left( \dot F_r \, \frac 1s \left( h(F_r + sr X) - h(F_r) \right)
    \right)\, dr,
  \end{align}
  where~$r \in [0, 1] \mapsto F_r$ is the straight line path
  connecting~$F_0$ and~$F$.  

  Let us consider the algebra~$\aN = L^\infty[0, 1] \bar \otimes \aM$
  equipped with the trace~$\tau_1 = \int_0^1 dr \otimes \tau$
  (see~\cite{Takesaki2002-MR1873025}).  The mapping~$\bar F : r \mapsto
  F_r$ is a $\tau_1$-Fredholm operator in~$\aN$ with~$\delta_{\bar F}
  \geq \delta$ and the mapping~$\bar X : r \mapsto rX$ is an element
  of~$\aN$.  Applying Lemma~\ref{FredholmContArgument}, we see that the
  mapping~$s \mapsto h(\bar F + s \bar X)$ is continuous in~$L^1(\aN)$
  in some neighbourhood of~$0$.  Consequently, letting~$s \rightarrow
  0$, yields that the first term in~(\ref{LastObjectiveI}) approaches $$
  \int_0^1 \tau \left( X\, h(F_r) \right)\, dr. $$ For the second term
  of~(\ref{LastObjectiveI}), we shall show that
  \begin{equation}
    \label{DOIinput}
    \lim_{s \rightarrow
      0} \int_0^1 \tau \left( \dot F_r\, \frac 1s \left( h(F_r + sr X) -
        h(F_r) 
      \right) \right)\, dr = \int_0^1 r\, \tau \left( X \, \frac d{dr}
      \left[ h(F_r) \right] \right) \, dr . 
  \end{equation}
  If~(\ref{DOIinput}) is proved, then, letting~$s \rightarrow 0$, we see
  from~(\ref{LastObjectiveI}) that $$ d \theta_F(X) = \tau \left( X
    \left( \int_0^1 h(F_r) + r \frac d{dr} \left[ h(F_r) \right] \, dr
    \right) \right). $$ Thus, to finish the proof
  of~(\ref{LastObjective}), we have to show that $$ \int_0^1 h(F_r) + r
  \, \frac {d}{dr} \left[ h(F_r) \right]\, dr = h(F). $$ This readily
  follows if we integrate the second term by parts.  Namely, integrating
  by parts, we have $$ \int_0^1 r \frac {d}{dr} \left[ h(F_r) \right] \,
  dr = h(F_1) - \int_0^1 h(F_r)\, dr. $$

  Next we prove~(\ref{DOIinput}).  The proof of~(\ref{DOIinput}) heavily
  relies on the theory of Double Operator Integrals (DOIs) developed
  in~\cite{PSW-DOI, PS-DiffP, PoSu} recently.  We will describe in
  Section 5 sufficient background on DOIs for the reader to appreciate
  their role in this Section.  Let again~$\aN = L^\infty[0,1] \bar
  \otimes \aM$ be a tensor product von Neumann algebra with the
  trace~$\tau_1 = \int_0^1 dr \otimes \tau$.  It is proved in
  Lemma~\ref{SmallLoopsSFcompletion} below that there are
  families\footnote{The property~(\ref{DOIdiff}) follows from the
    corresponding statement of Lemma~\ref{SmallLoopsSFcompletion} if one
    takes the group of $*$-automorphisms given by translations
    in~$L^\infty[0, 1] \bar \otimes \aM$, i.e.\ $$ \gamma_t(F_r) = F_{r
      + t},\ \ F_r \in \aN,\ r, t \in [0, 1] $$ where the group~$[0, 1]$
    is equipped with summation modulo~$1$\label{DiffFN}} of linear
  operators~$\left\{ T_s \right\}$, $\left\{ T'_s \right\}$ and~$\left\{
    T''_s \right\}$ uniformly bounded on~$\aN$ and on~$L^1(\aN)$ such
  that

  {  
    \renewcommand{\theenumi}{{\rm \alph{enumi}}}
    
    \begin{enumerate}
    \item $T_s = T_s' + T_s''$; \label{DOIresolution}
    \item $T_s'(Y) = g(F_r + srX) \, T'_s\left(Y \right)$, $Y \in \aN$;
      \label{DOIleftF}
    \item $T_s''(Y) = T''_s \left( Y\right) \, g(F_r) $, $Y \in \aN$;
      \label{DOIrightF}
    \item $h(F_r + sr X) - h(F_r) = T_s (sr X)$; \label{DOIperturbation}
    \item $\frac d{dr} \left[ h(F_r) \right] = T_0(\dot F_r)$;
      \label{DOIdiff}
    \item $\tau(T_0(Y) Z) = \tau(Y T_0(Z))$, $Y, Z \in \aN$;
      \label{DOIduality}
    \item $\lim_{s \rightarrow 0} \left\| T'_s(Y) - T'_0(Y) \right\|_1 =
      \lim_{s \rightarrow 0} \left\| T''_s(Y) - T''_0(Y) \right\|_1 = 0$,
      $Y \in L^1(\aN)$. \label{DOIlimit}
    \end{enumerate}
  }

  Observe first that~(\ref{DOIresolution}), (\ref{DOIleftF})
  and~(\ref{DOIrightF}) together with the fact that the mapping~$s
  \mapsto g(\bar F + s \bar X)$ is trace norm continuous, readily
  implies that~$T_s, T'_s, T''_s \in \Bd(\aN, L^1(\aN))$.  In
  particular, the fact that the operator~$T_0 \in \Bd(\aN, L^1(\aN))$
  and~(\ref{DOIdiff}) guarantee that the mapping~$r \mapsto \frac d{dr}
  \left[ h(F_r) \right]$ is trace norm continuous and the right hand
  side of~(\ref{DOIinput}) is well-defined.  Furthermore,
  \begin{multline*}
    \frac 1s \left( h(F_r + sr X) - h(F_r) \right)
    \stackrel{(\ref{DOIperturbation})} = T_s(r X) = \\
    \text{(\ref{DOIresolution}), (\ref{DOIleftF}) and~(\ref{DOIrightF})}
    \quad = g(F_r + srX) \, T'_s \left( rX \right) + T''_s \left( rX
    \right) \, g(F_r) \\ = \left( g(F_r + srX) - g(F_r) \right)\,
    T'_s \left( r X \right) + g(F_r) \, T'_s \left( rX \right) +
    T''_s (rX) \, g(F_r).
  \end{multline*}
  Letting~$s \rightarrow 0$ yields
  \begin{multline*}
    \lim_{s \rightarrow 0} \frac 1s \left( h(F_r + sr X) - h(F_r)
    \right) \stackrel{(\ref{DOIlimit})} = g(F_r)\, T'_0(rX) +
    T''_0(rX)\, g(F_r) \\ \text{(\ref{DOIleftF})
      and~(\ref{DOIrightF})\quad} = T'_0(rX) + T''_0(rX)
    \stackrel{(\ref{DOIresolution})} = T_0(rX),
  \end{multline*}
  where the limit converges in~$L^1(\aN)$.  Finally,
  \begin{multline*}
    \lim_{s \rightarrow 0} \tau_1 \left( \dot F_r\, \frac 1s \left(
        h(F_r + sr X) - h(F_r) \right) \right) \\
    = \tau_1 \left( \dot F_r \, T_0(rX) \right)
    \stackrel{(\ref{DOIduality})} = \tau_1 \left( X T_0(r \dot F_r)
    \right) \\ \stackrel{(\ref{DOIdiff})} = \tau_1 \left( X\, r \, \frac
      d{dr} \left[ h(F_r) \right] \right).
  \end{multline*}
  Thus,~(\ref{DOIinput}) is proved. \end{proof}

Let us now proceed with the proof of Theorem~\ref{thm:SFformulaI}.

\begin{proof}[Proof of Theorem~\ref{thm:SFformulaI}]
  Note that similarly to the proof of Theorem~\ref{SmallLoopsSF}, we may
  assume that~$\supp h \subseteq (-\delta, \delta)$.  

  Let us show that without loss of generality, we may assume that the
  path~$t \mapsto F_t$ is contained in~$\Fred_*^{\pm 1}$.  Indeed,
  consider a $C^2$-function~$\theta$ such that~$\supp \theta'
  \subseteq (-\delta, \delta)$ and~$\theta(\pm \delta) = \pm 1$.  Let
  us introduce the path~$t \mapsto G_t = \theta(F_t)$ and the
  function~$k(x) = K'(x)$, where~$K$ is such that~$H(x) = K(\theta
  (x))$.  Observe that~$G_t \in \Fred_*^{\pm 1}$ and~$\delta_{G_t} = 1$.

  Let us verify that the path~$t \mapsto G_t$ and the function~$k(x)$
  satisfies the assumptions of Theorem~\ref{thm:SFformulaI}.
  (\ref{SFformulaIi})~is clear from $$ 1 = \int_{-\delta}^\delta h(x) \,
  dx = \int_{-\delta}^\delta k(\theta(x))\, \theta'(x)\, dx =
  \int_{-1}^1 k(y)\, dy. $$ For~(\ref{SFformulaIii}), note that,
  (a)~since~$\supp k$ is a proper subinterval of~$(-1, 1)$, the
  function~$t \mapsto k(G_t)$ is bounded with respect to the trace
  norm~$\left\| \cdot \right\|_1$; (b)~according to
  Remark~\ref{OperNormCont}, the path~$t \mapsto G_t$ is $C^1$ with
  respect to the operator norm.  Consequently, the quantity in the
  assumption~(\ref{SFformulaIii}) is finite.  (\ref{SFformulaIiii}) is
  clear, since $$ K(G_j) = H(F_j) \ \ \text{and}\ \ \chi_{[0, +\infty)}
  (G_j) = \chi_{[0, + \infty)}(F_j),\ \ j = 0, 1. $$ Thus, we see that
  the path~$t \mapsto G_t$ and the function~$k(x)$ satisfies the
  assumption of Theorem~\ref{thm:SFformulaI} and~$G_t \in \Fred_*^{\pm
    1}$.

  On the other hand, we also have that $$ sf (F_t) = sf (G_t) \ \
  \text{and}\ \ \tau \left( \dot F_t \, h(F_t) \right) = \tau \left(
    \dot G_t \, k(G_t) \right),\ \ 0 \leq t \leq 1. $$ The former is
  clear from the definition of the spectral flow.  For the latter,
  observe that there is a uniformly bounded family of continuous linear
  operators~$\left\{ T_t \right\}_{0 \leq t \leq 1}$ on~$\aM$ such
  that~$\dot G_t = T_t(\dot F_t)$ (see
  Lemma~\ref{SmallLoopsSFcompletion}.(\ref{sDOIdiff})) such that $$ \tau
  \left( \dot G_t k(G_t) \right) =\tau \left( T_t(\dot F_t) k(\theta (F_t))
  \right)= \tau \left( \dot F_t \theta' (F_t) k (\theta (F_t)) \right) =
  \tau \left( \dot F_t h (F_t) \right), $$ where the second identity is
  due to Lemma~\ref{TraceLemma}\footnote{applied to the function~$$ \phi
    (\lambda, \mu) = k^{\frac 12} (\theta(\lambda))\, \frac
    {\theta(\lambda) - \theta(\mu)}{\lambda - \mu} \, k^{\frac 12}
    (\theta(\mu)). $$}.

  Therefore, for the rest of the proof, we assume that the path~$t
  \mapsto F_t$ is taken from~$\Fred_*^{\pm 1}$ and~$\supp h \subseteq
  (-1, 1)$.

  For every~$t \in [0, 1]$, let~$N_t$ be an open convex set given by $$
  N_t = \left\{ F\in \aM,\ \ \left\| \pi(F - F_t) \right\| < \epsilon_t
  \right\} \subseteq N(F_t), $$ for some~$0 < \epsilon_t < 1$ such
  that~$\supp h \subseteq (-\delta_t, \delta_t)$, where~$\delta_t =
  \min_{F \in N_t} \delta_F$ (one can find such~$N_t$ since the
  mapping~$F \mapsto \delta_F$ is continuous with respect to the
  semi-norm~$\left\| \pi\left( \cdot \right) \right\|$).  The preimages
  of the family~$\left\{ N_t \right\}_{0 \leq t \leq 1}$ under the
  mapping~$t \mapsto F_t$ produce an open covering of~$[0, 1]$.
  Consequently, due to compactness, we can {\it finitely\/} partition
  the segment~$[0, 1]$ by some points $$ 0 = t_0 < t_1 < \ldots < t_n =
  1 $$ such that every segment~$t \in [t_{k-1}, t_k] \mapsto F_t$ of the
  path~$t \mapsto F_t$ lies within the open convex set~$N_k = N_{t_k}
  \subseteq N(F_k)$ and~$\supp h \subseteq (-\delta_k, \delta_k)$,
  where~$\delta_k = \min_{F \in N_k} \delta_{F}$.  Observe also that
  identity~(\ref{SFformulaI}) (which we are proving) is additive with
  respect to partitioning of the path~$t \mapsto F_t$.  Thus, we need
  only to prove this identity for each segment~$[t_{k-1}, t_k]$.  Hence,
  from now on, we shall assume that the path~$t \in [0, 1] \mapsto F_t$
  lies entirely within the convex open set $$N = \left\{ F \in \aM,\ \
    \left\| \pi(F - F_0) \right\| < \epsilon \right\} \subseteq N(F)$$
  for some~$0 < \epsilon < 1$ and that~$\supp h \subseteq (-\delta,
  \delta)$, where~$\delta = \min_{F \in N} \delta_F$.
  
  Let~$B_j = 2\chi_{[0, +\infty)} (F_j) - 1$, $j= 0,1$ be two
  involutions and let~$t \in [0, 1] \mapsto B_t$ be the straight line
  path connecting~$B_0$ and~$B_1$ (i.e., $B_t = (1-t) B_0 + t B_1$).
  Since~$F_j \in \Fred_*^{\pm 1}$, by
  Lemma~\ref{FredholmProperty}.(\ref{TauCompact}), the difference~$F_j -
  B_j$ is $\tau$-compact and therefore~$\pi(F_j - B_j) = 0$.  The latter
  implies that the loop
  \begin{equation*}
    \begin{CD}
      F_0 @>>> F_1 \\
      @AAA @VVV \\
      B_0 @<<< B_1 \\
    \end{CD}
  \end{equation*}
  lies within the set~$N$, where the segment~$B_0 \rightarrow F_0$
  and~$F_1 \rightarrow B_1$ are the straight line paths.  Applying
  Theorem~\ref{SmallLoopsSF} for this loop implies that $$ \int_0^1 \tau
  \left( \dot B_t h (B_t) \right)\, dt = \int_0^1 \tau \left( \dot F_t h
    (F_t)\right)\, dt + \gamma_1 - \gamma_0, $$ where~$\gamma_j$ are the
  integrals along the straight line paths connecting~$F_j$ and~$B_j$,
  i.e., $$ \gamma_j = \int_0^1 \tau \left( \left( F_j - B_j \right)\, h
    \left( (1-t) F_j + t B_j\right) \right)\, dt,\ \ j = 0, 1. $$
  Let us show that 
  \begin{equation}
    \label{commutativeSF}
    H(F_j) - \frac 12 B_j = \int_0^1 (F_j - B_j) \, h\left( (1- t) F_j
      + t \, B_j \right)\, dt, \ \ j=0,1
  \end{equation}
  Observe that every operator in~(\ref{commutativeSF}) is a function
  of~$F_j$.  Moreover, the operators on both sides of this identity
  are supported on the projection~$E_j = \chi_{\supp h} (F_j)$.
  Indeed, on the right hand side, the support is determined by the
  function~$h$, and, on the left hand side, observe that the
  function $$ H - \frac 12 \, \left( \chi_{[0, +\infty)} -
    \chi_{(-\infty, 0)} \right) $$ vanishes outside of the
  support~$\supp h$, which clearly implies that~$H(F_j) -
  \frac 12 B_j$ is supported on~$E_j$.  Thus, we may consider the
  identity on the algebra generated by the operator~$E_j F_j$.  Since
  the projection~$E_j$ is $\tau$-finite, the latter algebra is
  $*$-isomorphic to a subalgebra~$L^\infty(\Rl, d \sigma)$,
  where~$\sigma(\Delta) = \tau \left( \chi_\Delta(F_j)\, E_j \right)$,
  $\Delta \subseteq \Rl$.

  In the setting of the algebra~$L^\infty(\Rl, d \sigma)$,
  identity~(\ref{commutativeSF}) holds a.e. due to the Newton-Leibniz
  theorem and the integral converges with respect to the ultra-weak
  topology.  This, in particular, implies that~$H(F_j) - \frac 12 B_j
  \in L^1(\Rl, d \sigma) \subseteq L^1(\aM)$.  Taking trace~$\tau$ from
  the latter identity gives $$ \gamma_j = \tau\left( H(F_j) - \frac 12
    B_j \right). $$

  Observing that from the definition of the spectral flow, it follows
  that~$sf(F_t) = sf (B_t)$, it is clear that to finish the proof we
  need only to show now that
  \begin{equation}
    \label{LastArgument}
    sf(B_t) = \int_0^1 \tau \left( \dot B_t \, h(B_t) \right)\, dt.
  \end{equation}
  The argument establishing~(\ref{LastArgument}) is similar
  to~\cite[Proposition~4.3]{Wahl}.  Observe, that the path~$B_t$
  consists of invertible operators excepting the point~$B_{\frac 12}$.
  Observe also that the operator~$B^2_{\frac 12}$ commutes with the
  every~$B_t$, $t \in [0, 1]$.  Let~$\delta_1$ be such that $0 <
  \delta_1 < \delta$ and~$\supp h \subseteq [-\delta_1,\delta_1]$.
  Since~$B_{\frac 12} \in N$, the projection~$E = \chi_{[0,
    {\delta_1^2}]} \left( B^2_{\frac 12} \right)$ is $\tau$-finite.
  Moreover, the projection~$E$ commutes with every~$B_t$, $t \in [0,
  1]$.  Let us decompose the path~$t \mapsto B_t$ into the direct sum of
  two paths
  \begin{equation}
    \label{PathDecomposition}
    t \mapsto E B_t E \ \ \text{and}\ \ t \mapsto (1- E) \, B_t\, (1 - E).
  \end{equation}
  Observe now that the second path in the latter decomposition consists
  of invertible operators (in the algebra~$(1- E)\, \aM \, (1 - E)$) and
  therefore the spectral flow vanishes on this path
  (see~\cite[Remark~2.3]{Phillips-SF2-1997}).  On the other hand, due to
  the choice of~$\delta_1$ and the projection~$E$, the spectrum of~$(1 -
  E) B_{\frac 12} (1 - E)$ lies outside of the interval~$[- \delta_1,
  \delta_1]$.  Furthermore, it is easy to see the inequality~$B^2_{t}
  \geq B^2_{\frac 12}$, which means that the statement about the
  spectrum of~$B_{\frac 12}$ above is equally valid for every~$(1 - E)\,
  B_t \, (1 - E)$.  Thus, we see that
  $$ h\left( (1 - E) B_{t} (1 - E) \right) = 0,\ \ t \in [0, 1] $$ and
  therefore the integral in~(\ref{LastArgument}) also vanishes on the
  second path in the decomposition~(\ref{PathDecomposition}).

  Identity~(\ref{LastArgument}) is additive with respect to
  direct sums.  Consequently, we need to prove~(\ref{LastArgument}) only
  for the path~$t \mapsto E B_t E$.  Regarding the latter path as a path
  in a finite algebra~$E\aM E$, the identity follows
  from~\cite[\S~5.1]{BCPRSW}.
\end{proof}
It now follows from the arguments in this Section that we have also proved the following result.
\begin{corollary}
Let $\Fred_\delta$ be the set of self adjoint $\tau$-Fredholm operators
in $\mathcal N$ whose essential spectrum does not intersect
the interval $[-\delta,\delta]$. This is an open submanifold of the Banach manifold of all self adjoint $\tau$-Fredholm operators.
For $h$ as in Theorem \ref{thm:SFformulaI} the one form $\theta$ on $\Fred_\delta$,
given by defining for each $F\in \Fred_\delta$ the functional 
$\theta_F$
on the tangent space to $\Fred_\delta$ at $F$ by $X\to\tau(Xh(F),$ is closed. Spectral flow along any piecewise $C^1$ path in $\Fred_\delta$ may be
interpreted as being obtained by integrating this one form.
\end{corollary}
\section{Spectral flow formula, unbounded case.}
\label{sec:integr-form-spectr}

We shall now discuss analytic formulae for paths of unbounded
self adjoint linear operators.  Across this Section,~$t \in [0, 1]
\mapsto D_t$ stands for a path of unbounded self adjoint linear
operators affiliated\footnote{Recall that a linear operator~$D: \dom(D)
  \mapsto \sH$ is called {\it affiliated with\/} a von Neumann
  algebra~$\aM$ if and only if~$u(\dom (D)) \subseteq \dom(D)$
  and~$u^*Du = D$ for every~$u \in \aM'$.} with~$\aM$.  In order to to
be able to compute the spectral flow of this path we assume that the
path~$t \in [0, 1] \mapsto F_t = \vartheta (D_t)$ is a continuous path
of $\tau$-Fredholm operators, where
\begin{equation}
  \label{VTfunction}
  \vartheta (x) = \frac x{(1 + x^2)^{\frac 12}}. 
\end{equation}
In this case, by definition, we set~$sf(D_t) = sf(F_t)$.  Furthermore,
to be able to consider analytic formulae for the spectral flow of the
path~$D_t$, we shall also impose a smoothness assumption onto~$D_t$.
Namely, the following definition is in order.

\begin{definition}
  \label{GammaSmoothness}
  \begin{enumerate}
  \item A path~$t\in [0, 1] \mapsto D_t$ is called
    $\Gamma$-differentiable at the point~$t = t_0$ if and only if there
    is a bounded linear operator~$G$ such that $$ \lim_{t \rightarrow
      t_0} \left\| \frac {D_t - D_{t_0}} t (1 + D_{t_0}^2)^{- \frac 12}
      - G \right\| = 0. $$ In this case, we set~$\dot D_{t_0} = G \, ( 1
    + D^2_{t_0})^{\frac 12}$.  The operator~$\dot D_t$ is a symmetric
    linear operator with the domain~$\dom (D_t)$ (see
    Lemma~\ref{LDclosedness} below).
  \item If the mapping~$t \mapsto \dot D_t (1 + D_t^2)^{- \frac 12}$ is
    defined and continuous with respect to the operator norm, then we
    call the path~$t \mapsto D_t$ continuously $\Gamma$-differentiable
    or~$C^1_\Gamma$-path\footnote{It may be shown that the class of all
      $C^1_\Gamma$-paths is the class of all paths which are
      continuously differentiable with respect to the graph norm of some
      fixed operator on this path (see~\cite{Wahl}).  We, however, will
      not use this connection below}.
  \end{enumerate}
\end{definition}

The main analytic spectral flow formula in the unbounded case is given
by the following theorem.

\begin{theorem}
  \label{thm:SFunbdd}
  Let~$t \in [0, 1] \mapsto D_t$ be a piecewise $C^1_\Gamma$-path of
  linear operators and~$\vartheta (D_t) \in \Fred_*^{\pm 1}$.  If~$g :
  \Rl \mapsto \Rl$ is a positive $C^2$-function such that
  \begin{enumerate}
  \item $\int_{-\infty}^{+\infty} g(x) \, dx = 1$;
  \item $\int_0^1 \left\| \dot D_t g(D_t) \right\|_1 \, dt < + \infty$;
  \item $G(F_1) - \frac 12 B_1 - G(F_0) + \frac 12 B_0 \in L^1(\aM)$,
    where~$B_j$ are the phases of\/~$D_j$, $j=0,1$, i.e., $B_j = 2
    \chi_{[0, + \infty)} (D_j) - 1$, and~$G$ is the antiderivative
    of~$g$ such that~$G(\pm \infty) = \pm \frac 12$;
  \end{enumerate}
  then $$ sf(D_t) = \int_0^1 \tau \left( \dot D_t\, g(D_t) \right)\, dt
  + \tau \left( G(D_1) - \frac 12 B_1 - G(D_0) + \frac 12 B_0 \right).
  $$
\end{theorem}

Since the spectral flow for a path of unbounded linear operators is
defined by the spectral flow of the corresponding path of bounded
operators (via the mapping~$D \mapsto \vartheta (D)$), the proof of the
theorem above is based on a reduction to the ``bounded'' spectral flow
formula given by Theorem~\ref{thm:SFformulaI}.  In this reduction the
main part is the question whether the path~$t \mapsto \vartheta(D_t)$
is~$C^1$ in the operator norm provided the path~$t \mapsto D_t$
is~$C^1_\Gamma$.  The latter question has been left open in~\cite{Wahl}
(see~p.~21).  We shall resolve this problem in
Theorem~\ref{ContDiffTheorem} below.

\begin{proof}[Proof of Theorem~\ref{thm:SFunbdd}]
  As in the proof of Theorem~\ref{thm:SFformulaI}, we may assume that
  the function~$g$ is compactly supported (see
  footnote~\ref{ToProperSubinterval} on
  p.~\pageref{ToProperSubinterval}).  Let~$h$ be a function such that $$
  g(x) = \frac {h(\vartheta (x))}{(1 + x^2)^{\frac 32}} $$ and let~$F_t
  = \vartheta (D_t)$.  The function~$h$ is a $C^2$-function supported on
  a proper subinterval of~$[-1, 1]$ and~$F_t \in \Fred_*^{\pm 1}$ by
  assumption.

  Let us verify that the path~$t \in [0, 1] \mapsto F_t \in \Fred_*^{\pm
    1}$ and the function~$h$ satisfies the hypothesis of
  Theorem~\ref{thm:SFformulaI}.  
  \begin{enumerate}
  \item Due to Theorem~\ref{ContDiffTheorem}, the mapping~$t \mapsto
    F_t$ is piecewise $C^1$.
  \item Observe that $$ \int_{-1}^1 h(\vartheta)\, d \vartheta =
    \int_{-\infty}^{+\infty} h(\vartheta (x))\, \vartheta ' (x)\, dx =
    \int_{-\infty}^{+\infty} g(x)\, dx = 1. $$
  \item Since~$g$ is compactly supported, the function~$h$ is supported
    on a proper subinterval of~$[-1, 1]$ and therefore the mapping~$t
    \mapsto h(F_t)$ is continuous in the trace norm (see
    Lemma~\ref{FredholmContArgument}).  In particular, $$ \int_0^1
    \left\| \dot F_t h(F_t) \right\|_1 \, dt < + \infty. $$
  \end{enumerate}
  Applying Theorem~\ref{thm:SFformulaI}, we readily obtain that $$
  sf(D_t) = \int_0^1 \tau\left( \dot F_t\, h(F_t) \right)\, dt + \tau
  \left( H(F_1) - H(F_0) - \frac 12 B_1 + \frac 12 B_0 \right). $$ Note
  that if~$H(x)$ is the antiderivative of~$h(x)$ such that~$H(\pm 1) =
  \pm \frac 12$, then~$H(\vartheta (x)) = G(x)$.  Consequently, $$
  H(F_j) = G(D_j),\ \ j = 0, 1 $$ and $$ \tau \left( \dot F_t \, h(F_t)
  \right) = \tau\left( \dot D_t \theta'(D_t)\, h(\vartheta (D_t))
  \right) = \tau \left( \dot D_t \, g(D_t) \right),\ \ t \in [0, 1]. $$
  The theorem is proved.
\end{proof}

Before we consider applications of Theorem~\ref{thm:SFunbdd}, let us
note that applying the argument of the proof above to
Theorem~\ref{SmallLoopsSF}, we obtain the answer to Singer's
question in the form framed for elliptic operators on compact manifolds. 

\begin{theorem}
  \label{SingerAnswer}
  If~$D$ is a self-adjoint linear operator with $\tau$-compact resolvent
  and~$g$ is a $C^2$-function, then the one form~$\tau \left( V g(D)
  \right)$ is exact on the affine space of all $D$-bounded
  perturbations~$V$ such that~$V g(D) \in L^1(\aM)$.  In other words,
  if~$t \in [ 0, 1] \mapsto D_t$ is a $C^1_\Gamma$-loop ($D_0 = D_1$) of
  unbounded self-adjoint linear operators with $\tau$-compact resolvent
  such that $$ \int_0^1 \left\| \dot D_t g(D_t) \right\|_1\, dt $$ is
  finite, then $$ \int_0^1 \tau \left( \dot D_t \, g(D_t) \right)\, dt =
  0. $$
\end{theorem}

By choosing specific functions~$g$, Theorem~\ref{thm:SFunbdd} above
allows a number of important corollaries.  We shall state only two of
them in Theorems~\ref{ThetaSumCorl} and~\ref{PSumCorl} below.  These
theorems extends \cite[Propositions~6.7 and~6.9]{Wahl} and earlier
results of~\cite{CarPhi2004-MR2053481, CarPhi1998-MR1638603}.

\begin{theorem}
  \label{ThetaSumCorl}
  If\/~$t \in [0, 1] \mapsto D_t$ is a piecewise $C^1_\Gamma$-path of
  unbounded linear operators such that~$D_t$ is
  $\theta$-summable\footnote{A self adjoint operator~$D$ is called
    $\theta$-summable if and only if~$e^{-\epsilon D^2} \in L^1(\aM)$
    for every~$\epsilon > 0$.}, $t \in [0, 1]$ and $$ \int_0^1 \left\|
    \dot D_t e^{- \epsilon D_t^2}\, \right\|_1 \, dt < + \infty,\ \
  \epsilon > 0, $$ then
  \begin{multline*}
    sf(D_0, D_1) = \sqrt {\frac \epsilon \pi} \int_0^1 \tau \left( \dot
      D_t \, e^{- \epsilon D_t^2} \right)\, dt \\ + \tau \left( G(D_1) -
      \frac 12 B_1 - G(D_0) + \frac 12 B_0 \right),\ \ \epsilon > 0,
  \end{multline*}
  where $$ G(x) = \sqrt {\frac \epsilon \pi} \int_{-\infty}^{x}
  e^{-\epsilon t^2} \, dt - \frac 12. $$
\end{theorem}

\begin{proof}
  The proof is specialization of Theorem~\ref{thm:SFunbdd} to the
  case~$g(x) = \sqrt {\frac \epsilon \pi} e^{- \epsilon x^2}$, provided
  we have checked that
  \begin{equation}
    \label{ThetaSumTmp}
    \vartheta (D_t) \in \Fred_*^{\pm 1},\  t \in [0, 1] \ \ \text{and}\
    \  G(D_j) - \frac 12 B_j \in L^1(\aM),\  j = 0,1.
  \end{equation}

  For the first statement in~(\ref{ThetaSumTmp}), observe that $$
  e^{-\epsilon n^2}\, \chi_{[-n, n]}(x) \leq e^{-\epsilon x^2},\ \ x \in
  \Rl,\ n \geq 1, $$ which means that every projection~$\chi_{[-n,
    n]}(D_t)$ is $\tau$-finite.  Furthermore, note also that, under the
  mapping~$x \mapsto \vartheta (x)$, compactly supported indicator
  functions are mapped onto indicators of proper subintervals of~$[-1,
  1]$.  Thus, we see that if~$F_t = \vartheta (D_t)$, then every
  projection~$\chi (F_t)$ is $\tau$-finite where~$\chi$ is an indicator
  of a proper subinterval of~$[-1, 1]$.  Consequently,~$F_t \in
  \Fred_*^{\pm 1}$.

  For the second statement in~(\ref{ThetaSumTmp}), let us consider the
  function 
  \begin{equation}
    \label{ThetaSumTmpF}
    f(x) = G(x) - \chi_{[0, + \infty)} (x) + \frac 12. 
  \end{equation}
  Clearly, $$ G(D_j) - \frac 12 B_j = f(D_j). $$ Thus, the required
  assertion follows from the estimate $$ \left| f(x) \right| \leq e^{-
    \frac \epsilon 2 x^2} \, \sqrt{\frac \epsilon \pi}\,
  \int_{-\infty}^{- \left| x \right|} e^{- \frac \epsilon 2 t^2}\, dt $$
  and the fact that~$D_t$ is $\theta$-summable.
\end{proof}

\begin{theorem}
  \label{PSumCorl}
  If\/~$t \in [0, 1] \mapsto D_t$ is a piecewise $C^1_\Gamma$-path such
  that~$D_t$ is $p$-summable\footnote{A self adjoint operator~$D$ is
    called $p$-summable, for some~$1 \leq p < \infty$ if and only
    if~$(1+D^2)^{-\frac p2} \in L^1(\aM)$} for some~$1 \leq p < \infty$
  and $$ \int_0^1 \left\| (1 + D_t^2)^{- \frac p2} \right\|_1 \, dt < +
  \infty, $$ then
  \begin{multline*}
    sf(D_0, D_1) = \frac 1{c_p} \int_0^1 \tau \left( \dot D_t \,
      (1+D_t^2)^{-\frac p2 - \frac 12} \right)\, dt \\ + \tau \left(
      G(D_1) - \frac 12 B_1 - G(D_0) + \frac 12 B_0 \right),
  \end{multline*} 
  where $$ c_p = \int_{-\infty}^{+\infty} (1 +
  x^2)^{-\frac p2 - \frac 12}\, dx,\ \ G(x) = \frac {1}{c_p}
  \int_{-\infty}^x (1 + t^2)^{- \frac p2 - \frac 12}\, dt - \frac 12. $$
\end{theorem}

\begin{proof}
  The proof consists of specialization of
  Theorem~\ref{thm:SFunbdd} (and justifying~(\ref{ThetaSumTmp})) for the
  case~$g(x) = \frac 1{c_p} (1 + x^2)^{-\frac p2 - \frac 12}$.

  By the assumption, the mapping~$t \mapsto \dot D_t (1 + D_t^2)^{-\frac
    12}$ is operator norm continuous, hence $$ \int_0^1 \left\| \dot D_t
    \, (1 + D_t^2)^{- \frac p2 - \frac 12} \right\|_1 \, dt < + \infty.
  $$

  Now, for the first statement in~(\ref{ThetaSumTmp}) as in the proof of
  Theorem~\ref{ThetaSumCorl} it is sufficient to note the estimate $$ (1
  + n^2)^{-\frac p2 - \frac 12} \, \chi_{[-n, n]} (x) \leq (1 +
  x^2)^{-\frac p2 - \frac 12},\ \ x \in \Rl,\ n \geq 1.  $$
  Consequently, every projection~$\chi_{[-n, n]}(D_t)$ is $\tau$-finite
  and the argument is repeated verbatim.

  For the second statement in~(\ref{ThetaSumTmp}), we shall estimate the
  function~$f(x)$ given in~(\ref{ThetaSumTmpF}) as follows
  \begin{multline*}
    \left| f(x) \right| = \frac 1{c_p} \, \int_{-\infty}^{-\left| x
      \right|} (1 + t^2)^{-\frac p2 - \frac 12}\, dt \leq \frac 1
    {c_p}\, \int_{-\infty}^{- \left| x \right|} \left| t \right|^{-p -
      1}\, dt \\ = \frac 1{p \, c_p}\, \left| x \right|^{-p} \leq \frac
    {2^{\frac p2 }}{p \, c_p}\, (1 + x^2)^{-\frac p2},\ \ \left| x
    \right| \geq 1.
\end{multline*}
\end{proof}

A customary assumption in non-commutative geometry (see~\cite{ACS,
  CarPhi1998-MR1638603, CarPhi2004-MR2053481}) is that the path~$t \in
[0, 1] \mapsto D_t$ is $C^1$ with respect to the operator norm (which is
a stronger assumption than the $C^1_\Gamma$ assumption).  Under this
assumption, the statement of Theorem~\ref{ThetaSumCorl} remains exactly
the same (except the symbol~$\dot D_t$ now stands for the ordinary
G\^ateaux derivative).  On the other hand, when~$t \mapsto D_t$ is a
$C^1$-path in the operator norm, Theorem~\ref{PSumCorl} changes to
Theorem~\ref{PSumCorlClass} below.  In the latter theorem, we no longer
need the additional resolvent factor under the trace in the spectral flow
formula to guarantee summability.  The observations above, regarding
piecewise $C^1$-paths $t \mapsto D_t$, cover the spectral flow formulae
proved in~\cite{ACS, CarPhi1998-MR1638603}.  Observe also that, in the
latter case, the $p$-summability assumption is no longer sufficient to
guarantee that the end points satisfy the boundary assumptions of
Theorem~\ref{thm:SFunbdd} and therefore we have to require this
explicitly in Theorem~\ref{PSumCorlClass} below.

\begin{theorem}
  \label{PSumCorlClass}
  If\/~$t \in [0, 1] \mapsto D_t$ is a piecewise $C^1$-path (with
  respect to the operator norm) such that~$D_t$ is $p$-summable for
  some~$1 \leq p < \infty$, $$ \int_0^1 \left\| (1 + D_t^2)^{- \frac p2}
  \right\|_1 \, dt < + \infty
  $$ and $$ G(D_1) - \frac 12 B_1 - G(D_0) + \frac 12 B_0 \in L^1(\aM),
  $$ then
  \begin{multline*}
    sf(D_0, D_1) = \frac 1{c_p} \int_0^1 \tau \left( \dot D_t \,
      (1+D_t^2)^{-\frac p2} \right)\, dt \\ + \tau \left( G(D_1) - \frac
      12 B_1 - G(D_0) + \frac 12 B_0 \right),
  \end{multline*}
  where $$ c_p = \int_{-\infty}^{+\infty} (1 + x^2)^{-\frac p2}\, dx,\ \
  G(x) = \frac {1}{c_p} \int_{-\infty}^x (1 + t^2)^{- \frac p2}\, dt -
  \frac 12. $$
\end{theorem}

\section{Double operator integrals.}
\label{sec:double-oper-integr}

In this section, we shall briefly outline the theory of double operator
integrals (DOI) , developed recently in~\cite{PS-RFlow, PSW-DOI, PS-DiffP,
  PoSu, PoSuNGapps}.  This theory unifies several different approaches
of harmonic analysis to smoothness properties of operator functions.  In
the present section, we shall mostly present the results (proved
somewhere else) needed to complete the proof of
Theorem~\ref{SmallLoopsSF} (see
properties~(\ref{DOIresolution})--(\ref{DOIlimit})) and those needed in
Section~\ref{sec:paths-self-adjoint}.

The theory of double operator integrals is a method of  giving an
integral representation of the difference $f(A)-f(B)$ where $f$ is a bounded
Borel function and $A$ and $B$ are self adjoint.
 In the case when $A,B$ are $n\times n$
matrices with spectral representation
 $A=\sum_{j=1}^n \lambda_jE_j$, $B=\sum _{j=1}^n\mu_kF_k$
 (here $E_j$ and $F_k$ denote spectral projections) 
this integral
representation is obtained from the following elementary computation
$$
f(A)-f(B)= \sum _{j,k=1}^n (f(\lambda_j)-f(\mu_k))E_jF_k=\sum _{j,k=1}^n
\frac{f(\lambda_j)-f(\mu_k))}{\lambda_j-\mu_k}E_j(A-B)F_k.
$$
In other words, we have just represented the difference $f(A)-f(B)$ as the
Stieltjes double operator integral $\int\int
\frac{f(\lambda)-f(\mu))}{\lambda-\mu} E_j(A-B)F_k.$ 
Notice that we are making use of the bimodule property of the $n \times n$ matrices.
An exposition of an early version of DOI which may assist the reader may be found in \cite{R}.

It is precisely the generalisation of this 
perturbation formula to infinite dimensional analogues that constitutes
the essence of the double operator integration theory initiated by Daletskii
and Krein and developed by Birman and Solomyak for type $I$ factors, and
further extended to semifinite  von Neumann algebras in
[17,18,16] and [31,32,34] to which we refer for additional historical information
and references.

Let~$\aM$ be a semi-finite von Neumann algebra and let~$\tau$ be a
n.s.f.~trace.  The symbol~$\sE$ stands for a non-commutative
fully\footnote{We shall omit the word ``fully'' in the sequel.}
symmetric ideal associated with the couple~$\left( \aM, \tau \right)$
(see~\cite{DDP1992, SukChi1990-MR1080637}).  In particular, $\sL^p$, $1
\leq p \leq \infty$ stands for the non-commutative $L^p$-Schatten ideal.
Furthermore, the symbol~$\sEcross$ stands for the K\"othe
dual~$\sEcross$ of a symmetric ideal~$\sE$ (see,~\cite{DDP1993}).  In
particular, if~$\sE = \sL^p$, $1 \leq p \leq \infty$, then~$\sEcross =
\sL^{p'}$, where~$\frac 1p + \frac 1{p'} = 1$.

We shall let $D_0, D_1$ denote
self adjoint unbounded operators
affiliated with $\aM$.  Let~$dE_\lambda^0$, $dE_\mu^1$ be the
corresponding spectral measures.  Recall that for every $K_1,K_2\in \sL^2$
$$ \tau(K_1\, dE_\lambda^0\,
K_2\, dE^1_\mu),\ \ \lambda, \mu \in \Rl $$ is a $\sigma$-additive
complex-valued measure on the plane~$\Rl^2$ with the total variation
bounded by~$\|K_1\|_{2} \|K_2\|_{2}$, see~\cite[Remark~3.1]{PSW-DOI}.

Let~$\phi = \phi(\lambda, \mu)$ be a bounded Borel function on~$\Rl^2$.
We call the function~$\phi$ {\it $dE^0 \otimes dE^1$-integrable\/} in
the symmetric ideal~$\sE$, if and only if there is
a linear operator~$T_\phi = T_\phi(D_0, D_1) \in B(\sE)$ such that
\begin{equation}
  \label{DOIdefFmla}
  \tau (K_1\, T_\phi(K_2)) = \int_{\Rl^2}
  \phi(\lambda, \mu)\, \tau(K_1\, dE_\lambda^0 \, K_2 \, dE_\mu^1),
\end{equation}
for every $$
K_1 \in \sL^2 \cap \sEcross \ \ \text{and}\ \ K_2 \in \sL^2
\cap \sE. $$
If the operator~$T_\phi(D_0, D_1)$ exists, then it is
unique,~\cite[Definition~2.9]{PSW-DOI}.  The latter definition is in
fact a special case of~\cite[Definition~2.9]{PSW-DOI}.  See
also~\cite[Proposition~2.12]{PSW-DOI} and the discussion there on
pages~81--82.  The operator~$T_\phi$ is called {\it the Double Operator
  Integral.}

We shall write~$\phi \in \Phi(\sE)$ if and only if the function~$\phi$
is $dE^0 \otimes dE^1$-integrable in the symmetric ideal~$\sE$ for any
measures~$dE^0$ and~$dE^1$.  

\begin{theorem}[{\cite{PSW-DOI, PS-DiffP}}]
  \label{HomomorphismResult}
  Let~$D_0, D_1$ be unbounded self adjoint operators affiliated to  $\aM$.  The mapping
  $$ \phi \mapsto T_\phi = T_\phi(D_0, D_1) \in \Bd(\sE),\ \ \phi
  \in \Phi(\sE) $$ satisfies $T_{\phi^\ast}=T_\phi^\ast$ and $T_{\phi\psi}= T_\phi T_\psi$.  Moreover,
  if~$\alpha, \beta: \Rl \mapsto \Cx$ are bounded Borel
  functions and 
  if~$\phi(\lambda, \mu) = \alpha(\lambda)$ (resp.\ $\phi(\lambda,
  \mu) = \beta(\mu)$), $\lambda, \mu \in \Rl$, then $$ T_\phi(K) =
  \alpha(D_0)\, K\ \ (resp.\ T_\phi(K) = K\, \beta(D_1)),\ \ K\in
  \sE.$$ 
\end{theorem}

The latter result allows the construction of  a sufficiently large class of
functions in~$\Phi(\sE)$.  Indeed, let us
consider the class~${\mathfrak{A}_0}$ which consists of all bounded
Borel functions~$\phi(\lambda, \mu)$, $\lambda, \mu \in \Rl$ admitting
the representation
\begin{equation}
  \label{AnotReps}
  \phi(\lambda, \mu) = \int_S \alpha_s(\lambda)\,
  \beta_s(\mu)\, d\nu(s)
\end{equation}
such that $$ \int_S \|\alpha_s\|_\infty \, \|\beta_s\|_\infty \, d\nu(s)
< \infty, $$ where~$(S, d\nu)$ is a measure space, $\alpha_s,
\beta_s:\Rl \mapsto \Cx$ are bounded Borel functions, for every~$s \in
S$ and~$\|\cdot\|_\infty$ is the operator norm.  The
space~$\mathfrak{A}_0$ is endowed with the norm $$
\|\phi\|_{\mathfrak{A}_0} := \inf \int_S \|\alpha_s\|_\infty \,
\|\beta_s\|_\infty \, d\nu(s), $$ where the minimum runs over all
possible representations~(\ref{AnotReps}).  The space~$\mathfrak{A}_0$
together with the norm~$\|\cdot\|_{\mathfrak{A}_0}$ is a Banach algebra,
see~\cite{PS-DiffP} for details.  The subspace of~$\mathfrak{A}_0$ of
all functions~$\phi$ admitting representation~(\ref{AnotReps}) with
continuous functions~$\alpha_s$ and~$\beta_s$ is denoted
by~$\mathfrak{C}_0$.  The following result is a straightforward
corollary of Theorem~\ref{HomomorphismResult}.

\begin{corollary}[{\cite[Proposition 4.7]{PS-DiffP}}]
  \label{AzeroClass}
  Every~$\phi \in {\mathfrak{A}_0}$ is $dE^0 \otimes dE^1$-integrable in
  the symmetric ideal~$\sE$ for any measures~$dE^0$, $dE^1$, i.e.\
  ${\mathfrak{A}_0} \subseteq \Phi(\sE)$.  Moreover, if\/~$T_\phi =
  T_\phi(D_0, D_1)$, for some self adjoint operators~$D_0, D_1,$ affiliated with
$  \aM$, then $$ \|T_\phi\|_{\Bd(\sE)} \leq \|\phi\|_{\mathfrak{A}_0}, $$
  for every~$\phi \in \mathfrak{A}_0$.
\end{corollary}

The major benefit delivered by the double operator integral theory is the
observation that, if~$D$ is a self adjoint linear operator 
affiliated with $\aM$ and~$A$ is a
self adjoint perturbation from~$\sE$, then the perturbation of the
operator function~$f(D)$ (where~$f: \Rl \mapsto \Cx$) is given by a
double operator integral.  Namely, $$ f(D + A) - f(D) = T_{\psi_f} (A),
$$ where~$T_{\psi_f}:= T_{\psi_f} (D+ A, D)\in B(\sE)$  and
\begin{equation}
  \label{PsiFfunction}
  \psi_f(\lambda, \mu) =
  \frac {f(\lambda) - f(\mu)}{\lambda - \mu},\ \lambda \neq \mu,\ \
  \psi(\lambda, \lambda) = f'(\lambda). 
\end{equation}
The identity above is proved in Theorem~\ref{PerturbationInM} below.
The proof is based on the following lemma which is a slight
generalization of~\cite[Lemma~7.1]{PSW-DOI}.  The proof of the lemma is
a repetition of that of~\cite[Lemma~7.1]{PSW-DOI} and therefore is
omitted.

\begin{lemma}
  \label{PerturbationLemma}
  Let~$D_j$ be self adjoint linear operators affiliated with
  $\aM$ with corresponding spectral measures~$E^j_n = E^{D_j}
  [-n, n]$, $j = 0,1$, $n = 1, 2,
  \ldots\,$.  If
  \begin{equation}
    \label{GeneralPerturbationPhi}
    \phi(\lambda_1, \lambda_0) = \frac
    {\beta_1(\lambda_1) \, \beta_0 (\lambda_0)} {\alpha_1 (\lambda_1)\,
      \alpha_0 (\lambda_0)} \, \psi_f (\lambda_1, \lambda_0)  \in
    \Phi(\sE), 
  \end{equation}
  where~$f$ is a Borel function and~$\alpha_j$, $\beta_j$ are bounded
  Borel functions, then, for every~$K \in \sE$,
  \begin{multline*}
    E^1_n \beta_1(D_1) \left[ f(D_1)\, K - K\, f(D_0) \right]\,
    \beta_0(D_0)\, E^0_n \\ = T_\phi \left( E_n^1 \alpha_1(D_1) \,
      \left[ D_1\, K - K\, D_0 \right] \, \alpha_0(D_0) \, E^0_n
    \right).
  \end{multline*}
\end{lemma}

\begin{theorem}
  \label{PerturbationInM}
  Let~$f$ be a Borel function and let~$\alpha_j$ and~$\beta_j$, $j=0,1$
  be bounded Borel functions.  Let~$D_j$, $j = 0, 1$ be self adjoint
  linear operators affiliated with $\aM$ and let~$A \in \aM$ be such that $$ B = \alpha_1 (D_1)
  \, \left[ D_1 \, A - A \, D_0 \right] \, \alpha_0(D_0) \in \aM. $$
  If\/~$\phi \in \Phi(\aM)$, where the function~$\phi$ is given
  in~(\ref{GeneralPerturbationPhi}), then
  $$ C = \beta_1(D_1)\, \left[ f(D_1)\, A - A\, f(D_0) \right]\,
  \beta_0(D_0) \in \aM $$ and $$ C = T_\phi(B). $$
\end{theorem}

\begin{proof}
  To be able to define the operator~$C$, first we have to ensure that
  \begin{equation}
    \label{DomainProp}
    A \beta_0 (D_0) [\dom (D_0)] \subseteq \dom (f(D_1) \, \beta_1(D_1)).
  \end{equation}
  To this end, we shall consider the bilinear form
  \begin{multline*}
    q(\xi, \eta) = \langle A \beta_0( D_0) \, \xi, \bar \beta_1 (D_1) \,
    \bar f(D_1) \, \eta \rangle \\ - \langle \beta_1 (D_1) \, A \beta_0
    (D_0) \, f (D_0) \, \xi, \eta \rangle, \ \ \xi \in \dom (D_0),\ \eta
    \in \dom (D_1).
  \end{multline*}
  Consider the spectral projections~$E^j_n = E^{D_j}[-n, n]$.  Let $$
  B_n = E^1_n \, \alpha_1 (D_1) \, \left[ D_1 \, A - A \, D_0 \right] \,
  \alpha_0(D_0) \, E^0_n $$ and $$ C_n = E^1_n \, \beta_1(D_1)\, \left[
    f(D_1)\, A - A\, f(D_0) \right]\, \beta_0(D_0) \, E^0_n. $$ We
  obtain from Lemma~\ref{PerturbationLemma} that $$ C_n = T_\phi (B_n).
  $$ Since the set of operators~$\{B_n\}_{n \geq 1}$ is uniformly bounded
  and~$\phi \in \Phi(\aM)$, this implies that the operators~$C_n$ are
  also uniformly bounded.  Furthermore, we see that $$ \lim_{n
    \rightarrow \infty} \langle C_n \xi, \eta \rangle = q(\xi, \eta),\ \
  \xi \in \dom (D_1),\ \eta \in \dom(D_0). $$ Consequently, the
  form~$q(\xi, \eta)$ is bounded and therefore we
  have~(\ref{DomainProp}).  Thus, the operator~$C$ is properly defined
  and bounded.  Moreover, since~$\dom(D_j)$, $j=0,1$ are dense in~$\sH$,
  we also have that $$ wo -\lim_{n \rightarrow \infty} C_n = C. $$
  Observe that we also have that $$ wo - \lim_{n \rightarrow \infty} B_n
  = B. $$ Since the operator~$T_\phi$ is continuous with respect to the
  weak operator topology (see~\cite[Lemma~2.4 and the proof of
  Proposition~2.6]{PoSu}), we finally obtain that $$ T_\phi (B) = C. $$
\end{proof}

\begin{remark}
  \label{OperNormCont} 
  It is clearly seen from Theorem~\ref{PerturbationInM} that, if~$\psi_f
  \in \Phi(\aM)$, then the function~$f$ maps (uniformly) operator norm
  continuous paths into themselves.  On the other hand, we know
  from~\cite[Theorem~4]{PoSuNGapps} that, for every function~$f: \Rl
  \mapsto \Cx$ such that $$ \left\| f \right\|_{\Lambda_\theta} +
  \left\| f' \right\|_{\Lambda_\epsilon} < + \infty,\ \ 0 \leq \theta <
  1,\ 0 < \epsilon \leq 1, $$ we have~$\psi_f \in \mathfrak{C}_0
  \subseteq \Phi(\aM)$.  Here~$\Lambda_\theta$ is the semi-norm on
  functions on $\Rl$ given by $$ \left\| f \right\|_{\Lambda_\theta} =
  \sup_{x_1, x_2} \frac {\left| f(x_1) - f(x_2) \right|}{\left| x_1 -
      x_2 \right|}. $$ Thus, in particular, every $C^2$-function maps
  operator norm (respectively, trace norm) continuous paths into
  operator norm (respectively, trace norm) continuous paths.
\end{remark}

Finally, we complete the proof of Theorems~\ref{SmallLoopsSF}
and~\ref{thm:SFformulaI} by establishing the following lemmas.

\begin{lemma}
  \label{SmallLoopsSFcompletion}
  Let~$g, h$ be compactly supported $C^2$-functions on $\Rl$ such
  that~$g(x) = 1$ for every~$x$ from some neighbourhood of~$\supp h$ and
  let~$r \mapsto \gamma_r$ be a weakly continuous group of
  $\tau$-invariant $*$-isomorphisms on~$\aM$ with the generator~$\delta: \dom (\delta)
  \mapsto \aM$, $\dom (\delta) \subseteq \aM$.  If~$F, X \in \aM$ are
  self adjoint, then there are families of linear operators~$\left\{ T_s
  \right\}$, $\left\{ T'_s \right\}$ and~$\left\{ T''_s \right\}$
  uniformly bounded on~$\aM$ and on~$L^1(\aM)$ such that

  {  
    \renewcommand{\theenumi}{{\rm \alph{enumi}}}
    
    \begin{enumerate}
    \item $T_s = T_s' + T_s''$; \label{sDOIresolution}
    \item $T_s'(Y) = g(F + srX) \, T'_s\left(Y \right)$, $Y \in \aM$;
      \label{sDOIleftF}
    \item $T_s''(Y) = T''_s \left( Y\right) \, g(F) $, $Y \in \aM$;
      \label{sDOIrightF}
    \item $h(F + s X) - h(F) = T_s (s X)$; \label{sDOIperturbation}
    \item if~$F \in \dom (\delta)$ and~$\lim_{r \rightarrow 0} \left\|
        \gamma_r(F) - F \right\| = 0$, then~$h(F) \in \dom (\delta)$
      and~$\delta(h(F)) = T_0(\delta (F))$;
      \label{sDOIdiff}
    \item $\tau(T_0(Y) Z) = \tau(Y T_0(Z))$, $Y, Z \in \aM$;
      \label{sDOIduality}
    \item $\lim_{s \rightarrow 0} \left\| T'_s(Y) - T'_0(Y) \right\|_1 =
      \lim_{s \rightarrow 0} \left\| T''_s(Y) - T''_0(Y) \right\|_1 = 0$,
      $Y \in L^1(\aM)$. \label{sDOIlimit}
    \end{enumerate}
  }
\end{lemma}

\begin{proof}
  We set~$T_s = T_{\psi_h} (F + sX, F)$.  It follows
  from~\cite[Corollary~7.6]{PS-DiffP} (see
  also~\cite[Theorem~4]{PoSuNGapps}) that~$\psi_h \in \mathfrak{C}_0$.
  Consequently, we readily see that~(\ref{sDOIperturbation}) follows
  from~\cite[Corollary~7.2]{PSW-DOI} (or Theorem~\ref{PerturbationInM});
  and~(\ref{sDOIduality}) --- from~\cite[Lemma~2.4]{PoSu}.

  Let~$g_1$ be a compactly supported $C^2$-function such that~$g_1(x) =
  1$ when~$x \in \supp h$ and~$g(x) = 1$ when~$x \in \supp g_1$.  We set
  $$ \psi_1(\lambda, \mu) = g_1(\lambda) \, \psi_h(\lambda, \mu) \ \
  \text{and}\ \ \psi_2 = \psi_h - \psi_1. $$ We also set~$T'_s =
  T_{\psi_1}(F + sX, F)$ and~$T''_s = T_{\psi_2} (F+ sX, F)$.  We
  instantly have~(\ref{sDOIresolution}).
  
  Note that~$\psi_1, \psi_2 \in \mathfrak{C}_0$.  Consequently,
  (\ref{sDOIlimit}) follows from~\cite[Lemma~5.14]{PS-DiffP}.

  We readily have from the construction that 
  \begin{equation}
    \label{PSIiProp}
    \psi_1(\lambda, \mu) =
    g(\lambda)\, \psi_1 (\lambda, \mu). 
  \end{equation}
  Furthermore, it may be observed that we also have
  \begin{equation}
    \label{PSIiiProp}
    \psi_2(\lambda, \mu) = \psi_2(\lambda, \mu)\,
    g(\mu). 
  \end{equation}
  Indeed, since the function~$h$ is compactly supported, the
  function~$\psi_h$ is supported in the cross~$S_1 \cup S_2 \subseteq
  \Rl \times \Rl$, where the strips~$S_j$, $j=1,2$ are given by $$ S_1 =
  \left\{ (\lambda, \mu):\ \ \lambda \in \supp h \right\},\ \ S_2
  \left\{ (\lambda, \mu):\ \ \mu \in \supp h \right\}. $$ By
  construction the function~$\psi_1$ coincides with~$\psi_h$ on the
  strip~$S_1$, i.e. $$ \psi_1(\lambda, \mu) = \psi_h(\lambda, \mu),\ \
  (\lambda, \mu) \in S_1. $$ Consequently, the function~$\psi_2 = \psi_h
  - \psi_1$ is supported within~$S_2$ which
  justifies~(\ref{PSIiiProp}).

  Clearly, (\ref{sDOIleftF}) and~(\ref{sDOIrightF}) follows from
  Theorem~\ref{HomomorphismResult} and~(\ref{PSIiProp})
  and~(\ref{PSIiiProp}).

  Let us show~(\ref{sDOIdiff}) (we refer the reader to~\cite{PS-RFlow,
    PotapovThesis} for a more complete study of the connection between double
  operator integrals and domains of derivations).  By definition 
  \begin{equation}
    \label{WeakGeneratorDomainDefinition}
    F \in \dom (\delta) \
    \ \Longleftrightarrow\ \ \lim_{t \rightarrow 0} \tau \left( \frac
      {\gamma_t(F) - F}{t} \, Y \right) = \tau \left( \delta(F)\, Y
    \right),\ \ Y \in L^1(\aM). 
  \end{equation}
  From~\cite[Corollary~7.2]{PSW-DOI}, we obtain that $$ h(\gamma_t(F)) -
  h(F) = S_t (\gamma_t(F) - F), $$ where~$S_t = T_{\psi_h} (\gamma_t(F),
  F)$.  Observe also that the family~$\left\{ S_t \right\}$ is different
  from~$\left\{ T_s \right\}$.  However,~$S_0 = T_0$.  Now,
  \begin{multline*}
    \tau \left( Y\, \left[ \frac {\gamma_t(h(F)) - h(F)}{t} -
        S_0(\delta(F)) \right]\right) \\ = \tau \left( Y\, S_0 \left[
        \frac {\gamma_t(F) - F}{t} - \delta(F) \right] \right) + \tau
    \left( Y\, \left[ \left( S_t - S_0 \right) \frac {\gamma_t (F) -
          F}{t} \right] \right) \\ = \tau \left( S_0^* (Y)\, \left[
        \frac {\gamma_t(F) - F}{t} - \delta(F) \right] \right) + \tau
    \left( (S_t^* - S_0^*) (Y)\, \frac {\gamma_t(F) - F}{t}  \right).
  \end{multline*}
  Letting~$t \rightarrow 0$, we see that the first term vanishes due to
  the fact that~$S_0^*$ is bounded on~$L^1(\aM)$ (see
  \cite[Lemma~2.4]{PoSu}) and~$F \in \dom (\delta)$.  Noting that the
  dual\footnote{Here, we consider dual operator~$S_t^*$ restricted
    on~$L^1(\aM) \subseteq \aM^*$, see details in~\cite{PoSu}.}
  family~$\left\{ S^*_t \right\}$ is a family of double operator
  integrals bounded on~$L^1(\aM)$ (see~\cite[Lemma~2.4]{PoSu}) and the
  family $$ \left\{ \frac {\gamma_t(F) - F}{t} \right\} $$ is uniformly
  bounded with respect to the operator norm
  (see~(\ref{WeakGeneratorDomainDefinition})), the second term vanishes
  due to \cite[Lemma~5.14]{PS-DiffP}.  Thus, according
  to~(\ref{WeakGeneratorDomainDefinition}) $h(F) \in \dom (\delta)$
  and~$\delta(h(F)) = T_0 (\delta(F))$.  
\end{proof}

\begin{lemma}
  \label{TraceLemma}
  Let~$\phi \in \uC_0$ such that~$\phi(\lambda, \mu) = \phi(\mu,
  \lambda)$, $\lambda, \mu \in \Rl$ and such that~$\phi$ is supported in
  a square~$I \times I$, where~$I$ is an interval.  Let~$D$ be a linear
  self adjoint operator affiliated with $\aM$ and let~$T = T_{\phi}(D, D)$ be a double
  operator integral (see~\cite{PS-DiffP}) which is a bounded linear
  operator on~$\aM$.  If\/~$E = \chi_I(D)$ is $\tau$-finite, then, for
  every~$V \in \aM$,
  \begin{equation}
    \label{TraceLemmaId}
    \tau\left( T(V)
    \right) = \tau \left( f(D)\, V \right),\ \ \text{where~$f(\lambda) =
      \phi(\lambda, \lambda)$}.
  \end{equation}
\end{lemma}

\begin{proof}
  Observe first that 
  \begin{equation}
    \label{TraceLemmaTmpI}
    T(1) = f(D).
  \end{equation}
  Indeed, if $\phi \in \uC_0$, then there is a measure space~$(S, \nu)$
  and continuous functions~$\alpha_s$ and~$\beta_s$, $s \in S$, such
  that $$ \phi(\lambda, \mu) = \int_S \alpha_s(\lambda)\, \beta_s(\mu)\,
  d\nu(s) $$ and $$ T(A) = \int_S \alpha_s(D) \, A \, \beta_s(D) \, d\nu
  (s), \ \ A \in \aM, $$ where $$ \int_S \left\| \alpha_s
  \right\|_\infty\, \left\| \beta_s \right\|_\infty \, d\nu(s) < +
  \infty. $$ Consequently,~(\ref{TraceLemmaTmpI}) follows from $$ T(1) =
  \int_S \alpha_s(D)\, 1 \, \beta_s(D)\, d\nu(s) = f(D). $$

  Since the operator~$T$ is self-dual (see~\cite[Lemma~2.4]{PoSu}), we
  obtain that $$ \tau \left( f(D) \, V \right) = \tau \left( f(D) \, E V
  \right) = \tau \left( T(1) \, E V \right) = \tau \left( T(EV) \right)
  = \tau \left( T(V) \right), $$ where the last identity follows from
  Theorem~\ref{HomomorphismResult}.  
\end{proof}

\section{Paths of self adjoint linear operators smooth in graph norm.}
\label{sec:paths-self-adjoint}

As we observed in Section~\ref{sec:integr-form-spectr}, analytic
spectral flow formulae for paths of unbounded self adjoint linear
operators are deduced from corresponding formulae
for paths of bounded Fredholm operators via the
mapping~$D \mapsto \vartheta (D)$, where the function~$\vartheta$ is
given in~(\ref{VTfunction}).  Consequently, the question of smoothness
properties of this mapping become of significant importance.  This
question has been studied deeply in~\cite{CPS, Sukochev2000-MR1767406,
  PoSuNGapps, CarPhi1998-MR1638603, CarPhi2004-MR2053481, Wahl}.

The main result of the present Section is that the function~$x \mapsto
\vartheta (x)$ maps $C^1_\Gamma$-paths onto~$C^1$-paths with respect to
the operator norm (see Theorem~\ref{ContDiffTheorem}).  This answers the
question asked in~\cite[p.~21]{Wahl}.  The proof is based on the
following observation, which is a development of the technique presented
in~\cite{PoSuNGapps}.  For every pair of self adjoint operators~$D_{j}$,
$j=0,1$, such that $D_1-D_0\in \Bd(\sH)$ there is a linear\footnote{In the present section, we shall
  consider only the operator norm.  Hence, we do not need an abstract
  von Neumann algebra.  Instead,~$\Bd(\sH)$ will suffice.} operator~$T$
on~$\Bd(\sH)$ such that
\begin{equation}
  \label{BasicIdea}
  F_1 - F_0 = T \left[ \left( D_1 - D_0
    \right)\, (1 + D_0^2)^{- \frac 12} \right], 
\end{equation}
where the operators~$F_j$ are given\footnote{When the operator~$D_1 -
  D_0$ is bounded and the resolvent of~$D_0$ is $\sE$-summable (i.e.,
  $\left( 1 + D_0^2 \right)^{- \frac 12} \in \sE$), then, since the
  operator~$T$ is bounded on the space~$\sE$, the
  identity~(\ref{BasicIdea}) yields that $$ \left\| F_1 - F_0
  \right\|_{\sE} \leq c\, \left\| D_1 - D_0 \right\| \, \left\| (1 +
    D_0^2)^{- \frac 12} \right\|_{\sE}. $$ The latter is proved
  in~\cite[Theorem~17]{PoSuNGapps} for an arbitrary symmetric
  ideal~$\sE$.  } by~$F_j = \vartheta (D_j)$, $j = 0,1$.  In the present
section, we shall further develop the above construction under the
weaker assumption that~$\dom (D_0) \subseteq \dom (D_1)$ and the
operator $$ (D_1 - D_0) \, (1 + D_0^2)^{-\frac 12} $$ is bounded which
is equivalent to the operator~$D_1 - D_0$ being bounded with respect to
the graph norm of the operator~$D_0$ (see Lemma~\ref{DboundedChar}
below).


Now let us state the two major results of the present Section.

\begin{theorem}
  \label{DiffTheorem}
  If\/~$\left\{ D_t \right\}$ is a collection of self adjoint operators
  $\Gamma$-differentiable (see Definition~\ref{GammaSmoothness}) at the
  point~$t = 0$, then the collection~$\left\{ F_t \right\}$, $F_t =
  \vartheta (D_t)$ is differentiable with respect to the operator norm
  at the point~$t = 0$.
\end{theorem}

\begin{theorem}
  \label{ContDiffTheorem}
  If~$t \in [0, 1] \mapsto D_t$ is a $C^1_\Gamma$-path (see
  Definition~\ref{GammaSmoothness}) of self adjoint linear operators,
  then~$t \in [0, 1] \mapsto \vartheta (D_t)$ is a $C^1$-path with
  respect to the operator norm, where the function~$\vartheta$ is given
  by~(\ref{VTfunction}).
\end{theorem}

\subsection{$D$-bounded operators.}
\label{sec:d_0-bound-oper}

Let~$\sH$ be Hilbert space and let~$D: \dom (D) \mapsto \sH$ be a linear
operator with the domain~$\dom (D) \subseteq \sH$.  A linear operator~$A
: \dom (A) \mapsto \sH$ ($\dom (A) \subseteq \sH$) is called $D$-bounded
if and only if~$\dom (D) \subseteq \dom(A)$ and there is a
constant~$c>0$ such that
\begin{equation}
  \label{DboundedDefEst}
  \left\| A (\xi)
  \right\|_{\sH} \leq c \, \left( \left\| \xi \right\|_{\sH}^2 + \left\|
      D(\xi) \right\|_{\sH}^2 \right)^{\frac 12},\ \ \xi \in \dom (D). 
\end{equation}
We let~$\left\| A \right\|_D$ be the smallest possible constant~$c> 0$
such that~(\ref{DboundedDefEst}) holds.

Observe that if an operator~$A$ is~$D_0$-bounded and~$\left\| A
\right\|_{D_0} < 1$, then the operator~$A$ is also $D$-bounded, where~$D
= A + D_0$.  Indeed, suppose that~$0 < c < 1$ is the constant such
that~(\ref{DboundedDefEst}) holds.  It then follows that $$ \left\| A
  (\xi) \right\|_{\sH} \leq c \, \left( \left\| \xi \right\|_{\sH} +
  \left\| D_0 (\xi) \right\|_{\sH} \right) \leq c \, \left( \left\| \xi
  \right\|_{\sH} + \left\| D(\xi) \right\|_{\sH} + \left\| A(\xi)
  \right\|_{\sH} \right). $$ This implies that $$ \left\| A (\xi)
\right\|_{\sH} \leq \, \frac c{1 - c} \, \left( \left\| \xi
  \right\|_{\sH} + \left\| D(\xi) \right\|_{\sH} \right) \leq \frac {c
  \, \sqrt 2}{1 - c} \, \left( \left\| \xi \right\|_{\sH}^2 + \left\|
    D(\xi) \right\|_{\sH}^2 \right)^{\frac 12}. $$ In other words, 
\begin{equation}
  \label{DnotToDbdd}
  \left\| A \right\|_D \leq \frac {\sqrt 2 \, \left\| A \right\|_{D_0}}
  {1 - \left\| A \right\|_{D_0}}. 
\end{equation}

Furthermore, if~$D_j: \dom(D_j) \mapsto \sH$, $j=0,1$ and~$A: \dom (A)
\mapsto \sH$ are linear operators and~$\left\| D_1 - D_0 \right\|_{D_0}
< 1$, then~$A$ is $D_0$-bounded if and only if~$A$ is $D_1$-bounded.
Indeed, firstly since~$D_1 - D_0$ is $D_0$-bounded, we see that~$\dom
(D_0) = \dom (D_1)$.  Secondly, according to~(\ref{DnotToDbdd}), $D_1 -
D_0$ is also $D_1$-bounded.  Finally, if~$c$ is the $D_0$-norm of~$A$
and~$c'$ is the $D_1$-norm of~$D_1 - D_0$, then
\begin{multline*}
  \left\| A(\xi) \right\|_{\sH} \leq c\, \left( \left\| \xi
    \right\|_{\sH} + \left\| D_0(\xi) \right\|_{\sH} \right) \\ \leq c\,
  \left( \left\| \xi \right\|_{\sH} + \left\| D_1(\xi) \right\|_\sH +
    \left\| (D_1 - D_0) (\xi) \right\|_\sH \right) \\ \leq c \, \left(
    \left\| \xi \right\|_\sH + \left\| D_1 (\xi) \right\|_\sH + c' \,
    \left( \left\| \xi \right\|_\sH + \left\| D_1 (\xi) \right\|_\sH
    \right) \right) \\ \leq \sqrt 2 \, c \, (1 + c') \, \left( \left\|
      \xi \right\|^2_\sH + \left\| D_1 (\xi) \right\|^2_\sH
  \right)^{\frac 12},\ \ \xi \in \dom (D_0).
\end{multline*}
Thus, if~$A$ is~$D_0$-bounded then~$A$ is $D_1$-bounded.  The opposite
implication is similar.  In other words, we proved that
\begin{equation}
  \label{Dxequiv}
  c_1\,  \left\| A \right\|_{D_0} \leq \left\| A \right\|_{D_1} \leq c_2
  \, \left\| A \right\|_{D_0},
\end{equation}
where~$c_1$ and~$c_2$ are positive constants depending on~$\left\| D_1 -
  D_0\right\|_{D_0} < 1$.

\begin{lemma}
  \label{DboundedChar}
  Let~$D: \dom(D) \mapsto \sH$ be a self adjoint linear operator and
  let~$A : \dom (A) \mapsto \sH$ be a linear operator such that~$\dom
  (D) \subseteq \dom (A)$.  The following are equivalent:
  \begin{enumerate}
  \item $A$ is $D$-bounded; \label{DCi}
  \item $A \, (i + D)^{-1}$ and~$A\, (-i + D)^{-1}$ are bounded;
    \label{DCii}
  \item $A \, (1 + D^2)^{- \frac 12}$ is bounded. \label{DCiii}
  \end{enumerate}
\end{lemma}

\begin{proof}
  From~(\ref{DCi}) to~(\ref{DCii}).  If~$c > 0$ is a constant such
  that~(\ref{DboundedDefEst}) holds, then, since $$ \left( \pm i + D
  \right)^{-1} (\sH) \subseteq \dom (D), $$ we have that
  \begin{multline*}
    \left\| A \, (\pm i + D)^{-1} (\xi) \right\|_{\sH} \\ \leq c \,
    \left( \left\| (\pm i + D)^{-1} (\xi) \right\|^2_{\sH} + \left\| D
        \, (\pm i + D)^{-1} (\xi) \right\|^2_\sH \right)^{\frac 12} \\
    \leq c \, \sqrt 2 \, \left\| \xi \right\|_{\sH},\ \ \xi \in \sH.
  \end{multline*}
  This means that the operator~$A \, (\pm i + D)^{-1}$ is bounded.

  From~(\ref{DCii}) to~(\ref{DCi}).  Let~$c = \left\| A \, (\pm i +
    D)^{-1} \right\| < + \infty$.  We instantly obtain that
  \begin{multline*}
    \left\| A(\xi) \right\|_\sH = \left\| A \, (\pm i + D)^{-1} (\pm i +
      D) (\xi) \right\|_\sH \\ \leq c \, \left\| (\pm i + D) (\xi)
    \right\|_{\sH} \leq c \, \left( \left\| \xi \right\|_\sH + \left\|
        D(\xi) \right\|_\sH \right) \\ \leq c \, \sqrt 2\, \left(
      \left\| \xi \right\|^2_\sH + \left\| D(\xi) \right\|_\sH^2
    \right)^{\frac 12},\ \ \xi \in \dom (D),
  \end{multline*}
  which means that~$A$ is $D$-bounded.

  The equivalence of~(\ref{DCii}) and~(\ref{DCiii}) follows from the
  fact that the operators $$ (\pm i + D) (1 + D^2)^{- \frac 12} \ \
  \text{and}\ \  (1 + D^2)^{\frac 12} (\pm i + D)^{-1} $$ are unitary.
\end{proof}

\begin{remark}
  \label{Isomorphism}
  It follows from the proof of Lemma~\ref{DboundedChar} that $$ \frac 1
  {\sqrt 2}
  \, \left\| A \right\|_{D} \leq \left\| A \, (1 + D^2)^{-\frac 12}
  \right\| \leq \sqrt 2\, \left\| A \right\|_{D}. $$
\end{remark}

Observe that, according to Lemma~\ref{DboundedChar} and
Remark~\ref{Isomorphism}, a path~$t \mapsto D_t$ is
$\Gamma$-differentiable at the point~$t = 0$ (as defined in
Section~\ref{sec:integr-form-spectr}) if and only if $\dom (D_0)
\subseteq \dom (D_t)$ in some neighbourhood of~$t = 0$ and $$ \lim_{t
  \rightarrow 0} \left\| \frac {D_t - D_0}t - \dot D_0 \right\|_{D_0} =
0, $$ in other words, if and only if the path~$t \mapsto D_t$ is
differentiable with respect to the graph norm of the operator~$D_0$ at
the point~$t = 0$.  This observation further extends to



\subsection{The subspace~$L(D)$.}
\label{sec:subspace-ld}

Recall that an operator~$A: \dom(A) \mapsto \sH$ is called {\it
  symmetric\/} if and only if $$ \langle A(\xi), \eta \rangle = \langle
\xi, A(\eta) \rangle,\ \ \xi, \eta \in \dom (A). $$ Let~$D: \dom (D)
\to \sH$ be a self adjoint linear operator and let~$L(D)$ be the real
linear space  consisting of all symmetric $A:
\dom(A) \mapsto \sH$ where
$\dom(D) \subseteq \dom (A)$ and such that\footnote{Here~$A \, (1 + D^2)^{-\frac
    12} \in \Bd(\sH)$ means that the operator~$A \, (1 + D^2)^{-\frac
    12}$ is closable and the closure belongs to~$\Bd(\sH)$.} 
    $A \, (1 + D^2)^{-\frac 12} \in \Bd(\sH)$.

Next, observe that, since~$A$ is symmetric and~$(1 + D^2)^{- \frac 12}$ is
self adjoint, we have 
\begin{equation*}
  \langle (1 + D^2)^{- \frac 12} A (\xi), \eta \rangle = \langle A
  (\xi), (1 + D^2)^{- \frac 12} (\eta) \rangle  =  \langle \xi, A
  \, (1 + D^2)^{- \frac 12} (\eta) \rangle.
\end{equation*}
Consequently, we have the implication
\begin{equation}
  \label{LeftToRight}
  A \, (1 + D^2)^{-\frac 12} \in \Bd(\sH) \ \ \Longrightarrow \ \ (1 +
  D^2)^{-\frac 12} A \in \Bd(\sH).
\end{equation}
More precisely, if the operator~$A \, (1 + D^2)^{- \frac 12}$ is
bounded, then the operator~$(1 + D^2)^{-\frac 12} A$ is closable and its
closure is also bounded.

\begin{lemma}
  \label{LDclosedness}
  $L(D)$ is a closed subspace of\/~$\Bd(\sH)$.
\end{lemma}

\begin{proof}
  We prove that the subspace~$L(D)$ is closed with respect to the weak
  operator topology.  Let~$B_n \in L(D)$, $n\ge 1$,  and $$ wo-\lim_{n \rightarrow
    \infty} B_n = B \in \Bd(\sH). $$ This means that there is a
  sequence~$A_n$ of symmetric operators such that $$ B_n = A_n (1 +
  D^2)^{- \frac 12} \ \ \text{and}\ \ \dom(D) \subseteq \dom (A_n). $$
  We need to show that~$B \in L(D)$.  Introduce the operator~$A: \dom
  (A) \mapsto \sH$ by setting $$ \dom(A) = \dom(D) \ \ \text{and}\ \ A =
  B \, (1 + D^2)^{\frac 12}. $$ We clearly have that the operator~$A \,
  (1 + D^2)^{- \frac 12}$ is closable and its closure coincides with~$B$
  i.e., $$ B = \overline{A \, (1 + D^2)^{- \frac 12}}. $$ Consequently, we have
  only to verify that the operator~$A$ is symmetric.  To this end,
  observe that
  \begin{align*}
    \langle A (\xi), \eta \rangle = &\, \langle B \, (1 + D^2)^{\frac
      12} (\xi), \eta \rangle \\ = &\, \lim_{n \rightarrow \infty}
    \langle B_n\, (1 + D^2)^{\frac 12} (\xi), \eta \rangle \\ = &\,
    \lim_{n \rightarrow \infty} \langle A_n (\xi), \eta \rangle \\ = &\,
    \lim_{n \rightarrow \infty} \langle \xi, A_n(\eta)\rangle \\ = &\,
    \lim_{n \rightarrow \infty} \langle \xi, B_n \, (1 + D^2)^{\frac 12}
    (\eta) \rangle \\ = &\, \langle \xi, B \, (1 + D^2)^{\frac
      12}(\eta)\rangle \\ = &\, \langle \xi, A(\eta) \rangle,\ \ \xi, \eta \in
    \dom (A).
  \end{align*}
  Thus, $A$ is symmetric and therefore~$B \in L(D)$.  
\end{proof}

For a closely related argument to the following see \cite{Lesch}.
\begin{lemma}
  \label{ToSymLemma}
  Let~$D_j : \dom (D_j) \mapsto \sH$, $j= 0,1$ be a self adjoint linear
  operator and let~$B \in L(D_0)$.  If\/~$D_1 - D_0$ is $D_0$-bounded
  and~$\left\| D_1 - D_0 \right\|_{D_0} < 1$, then the
  operator~$B_\theta = (1 + D_1^2)^{-\frac \theta 2} B \, (1 +
  D_0^2)^{\frac \theta 2}$ is bounded and $$ \left\| B_\theta \right\|
  \leq c_0\, \left\| B \right\|, $$ for some constant~$c_0 > 0$ and
  every~$0 \leq \theta \leq 1$.
\end{lemma}

\begin{proof}
  Let~$A : \dom (A) \mapsto \sH$ be a symmetric linear operator such
  that
  \begin{equation}
    \label{ToSymLemmaRight}
    B = A \, (1 + D_0^2)^{- \frac 12} \in
    \Bd(\sH). 
  \end{equation}
  In particular, $$ \dom (D_0) \subseteq \dom (A). $$ The operator~$A$
  is $D_0$-bounded (see Lemma~\ref{DboundedChar}).  According
  to~(\ref{Dxequiv}), the operator~$A$ is also $D_1$-bounded.  This
  further means (using~(\ref{LeftToRight}) and Lemma~\ref{DboundedChar})
  that
  \begin{equation}
    \label{ToSymLemmaLeft}
    (1 + D_1^2)^{- \frac 12} A \in
    \Bd(\sH). 
  \end{equation}
  Let~$E^{j}_{n} = E^{D_j}(-n, n)$ be the spectral projection of the
  operator~$D_j$, $j=0,1$.  The operator~$E^{1}_{ n} A E^{0}_{ n}$ is
  bounded since $$ E^{1}_{n} A E^{0}_{ n} = E^1_n\, B \, (1 +
  D_0^2)^{\frac 12} \, E^{0}_{ n}. $$ Let~$$ C_{\theta, n} = (1 +
  D_1^2)^{- \frac \theta 2} E^{1}_{ n} B E^{0}_{ n} (1 + D_0^2)^{\frac
    \theta 2}. $$ We clearly have that $$ \lim_{n \rightarrow \infty }
  \langle C_{\theta, n} (\xi) , \eta \rangle = \langle B_\theta(\xi),
  \eta\rangle,\ \ \xi \in \dom(D_0),\ \eta \in \dom(D_1). $$ Thus, it is
  sufficient to show that the operators~$C_{n, \theta}$ are uniformly
  bounded with respect to~$n$.  This follows
  from~(\ref{ToSymLemmaRight}), (\ref{ToSymLemmaLeft}) and
  Lemma~\ref{MaxPrincipleCor} below.
\end{proof}

\begin{lemma}
  \label{MaxPrincipleCor}
  Let~$A$, $B_j$, $j= 0,1$ be bounded linear operators.  If~$B_j$, $j =
  0,1$ are positive, then the operator~$B_1^{1-\theta} A \,
  B_0^{\theta}$ is bounded and
  \begin{equation}
    \label{MaxPrincipleCorEst}
    \left\| B_1^{1 - \theta}
      A\, B_0^{\theta} \right\| \leq \left\| B_1 A \right\|^{1 - \theta}
    \, \left\| A B_0 \right\|^{\theta},\ \ 0 \leq \theta \leq 1. 
  \end{equation}
\end{lemma}

\begin{proof}
  The lemma is a straightforward application of the three lines lemma
  (see~\cite[Lemma~1.1.2]{BergLofstrom}) to the
  holomorphic function $$ f_{\xi, \eta}(z) = \left\| B_1 A
  \right\|^{-1 + z} \left\| A \, B_0 \right\|^{-z} \, \langle B_1^{1 -
    z} A\, B_0^{z} (\xi), \eta \rangle,\ \ \xi, \eta \in \sH, $$
  considered in the strip~$S = \left\{z \in \Cx:\ 0 < \Re z < 1
  \right\}$.\end{proof}

\subsection{The proof of Theorems~\ref{DiffTheorem}
  and~\ref{ContDiffTheorem}.}
\label{sec:proof-main-results}

The proof of Theorems~\ref{DiffTheorem} and~\ref{ContDiffTheorem} rests
on the properties of the operator~$T_\phi (D_1, D_0)$ where the
function~$\phi$ is given by
\begin{equation}
  \label{PhiDef}
  \phi(\lambda, \mu) = \frac {\vartheta(\lambda) -
    \vartheta(\mu)}{\lambda - \mu} \, (1 + \mu^2)^{\frac 12}. 
\end{equation}
The difficulty about this operator is the fact that it is not bounded
on~$\Bd(\sH)$.  Thus, a direct application of the methods of double operator
integrals and harmonic analysis is not feasible.  Nevertheless, we shall
show that this operator, when considered on the subspace~$L(D_0)$, is
bounded and possesses all the properties needed to prove
Theorems~\ref{DiffTheorem} and~\ref{ContDiffTheorem}.

Let~$D_j : \dom (D_j) \mapsto \sH$, $j = 0,1$ be a self adjoint linear
operator such that~$D_1 - D_0$ is $D_0$-bounded and~$\left\| D_1 - D_0
\right\|_{D_0} \leq \frac 12$.  In order to introduce the
operator~$T_\phi = T_\phi(D_1, D_0) : L(D_0) \mapsto \Bd(\sH)$ ($\phi$
is given in~(\ref{PhiDef})), let us consider another function
\begin{equation}
  \label{PsiDef}
  \psi(\lambda, \mu) = (1 + \lambda^2)^{\frac 14} \, \frac
  {\vartheta(\lambda) - \vartheta(\mu)}{\lambda - \mu} \, (1 +
  \mu^2)^{\frac 14}.  
\end{equation}
The operator~$T_\psi = T_{\psi} (D_1, D_0)$ is bounded on~$\Bd(\sH)$,
($T_\psi$ is equal to the operator~$T_{\theta}$ with~$\theta = \frac 12$
introduced in the proof of~\cite[Theorem~14]{PoSuNGapps}).  Observe also
that the bound of operator~$T_\psi$ does not depend of the
operators~$D_j$, $j = 0,1$.

The fact that the operator~$T_\psi$ is bounded on~$\Bd(\sH)$ and
Lemma~\ref{ToSymLemma} imply that the mapping
\begin{equation}
  \label{TphiDef}
  B \in L(D_0) \mapsto
  T_\psi ((1 + D_1^2)^{- \frac 14} B\, (1 + D_0^2)^{\frac 14}) \in \Bd(\sH)
\end{equation}
is a bounded linear operator~$L(D_0) \mapsto \Bd(\sH)$ whose norm
depends only on the quantity~$\left\| D_1 - D_0 \right\|_{D_0}$.
Furthermore, it is known (see~\cite{PoSu}, see also~\cite{PotapovThesis}
for a more complete and detailed exposition) that the operators~$T_\phi$ and~$T_\psi$
are bounded from~$\cC^2 \mapsto \cC^2$, where~$\cC^2 \subseteq \Bd(\sH)$ is
the Hilbert-Schmidt ideal and the following identity holds

$$ T_\phi(B)\, (1+D_0^2)^{- \frac 14} = (1 + D_1^2)^{- \frac 14 }T_\psi
(B),\ \ B \in \cC^2. $$ The latter identity suggests that the
mapping~(\ref{TphiDef}) is a (unique) bounded extension of the
operator~$T_\phi$ from~$\cC^2 \cap L(D_0)$ to the space~$L(D_0)$.
Motivated by this observation, we shall write~$T_\phi = T_\phi(D_1,
D_0)$ for the mapping~(\ref{TphiDef}).

\begin{proof}[Proof of Theorem~\ref{DiffTheorem}]
  Let~$T_{\phi, t} = T_\phi (D_t, D_0)$ and let~$H = T_{\phi, 0} (G)$
  where~$G = \dot D_0 \, (1 + D_0^2)^{-\frac 12}$ (observe that the
  subspace~$L(D_0)$ is closed in~$\Bd(\sH)$ and therefore~$G \in
  L(D_0)$).  It now follows from Theorem~\ref{PerturbationInM} that
  \begin{equation*}
    \frac {F_t - F_0} t - H = T_{\phi, t} \left( \frac {D_t - D_0} t \,
      (1 + D_0^2)^{-\frac 12} - G \right) + \left( T_{\phi, t} (G) -
      T_{\phi, 0}  (G) \right). 
  \end{equation*}
  When~$t \rightarrow 0$, the first term vanishes due to assumptions of
  the theorem and the fact that the operators~$T_{\phi, t}$ are
  uniformly bounded.  To finish the proof of the theorem, we need to
  justify that
  \begin{equation}
    \label{DiffTheoremAim}
    \lim_{t \rightarrow 0} \left\|
      T_{\phi, t} (G) - T_{\phi, 0} (G) \right\| = 0. 
  \end{equation}
  Letting~$T_{\psi, t} = T_{\psi} (D_t, D_0)$ and $$ C_t = (1 +
  D_t^2)^{- \frac 14} G \, (1 + D_0^2)^{\frac 14}, $$ we infer
  (via~(\ref{TphiDef})) that
  \begin{equation*}
    T_{\phi, t} (G) - T_{\phi, 0} (G) =  T_{\psi, t} (C_t) - T_{\psi,
      0} (C_0)  =  T_{\psi, t} (C_t - C_0) + (T_{\psi, t} (C_0) -
    T_{\psi, 0} (C_0)).
  \end{equation*}
  Observing that~$T_{\psi, t}$ are uniformly bounded, we see that it is
  sufficient to show
  \begin{equation}
    \label{DiffTheoremAimI}
    \lim_{t \rightarrow 0} \left\| C_t - C_0 \right\| = 0
    \ \ \text{and}\ \ \lim_{t \rightarrow 0} \left\| T_{\psi, t} (C_0) -
      T_{\psi, 0} (C_0) \right\| = 0. 
  \end{equation}
  The first limit in~(\ref{DiffTheoremAimI}) is due to the following
  identity
  \begin{align*}
    C_t - C_0 = &\, (1 + D_t^2)^{- \frac 14} G \, (1 + D_0^2)^{\frac 14}
    - (1+D_0^2)^{- \frac 14} G\, (1 + D_0^2)^{\frac 14} \cr = &\, \left(
      (1+ D_t^2 )^{- \frac 14} (1+ D_0^2)^{\frac 14} - 1 \right) \,
    C_0
  \end{align*}
  and the following estimate~(see
  Lemma~\ref{AuxilliaryEsts}.(\ref{AuxilliaryEstI}) below)
  \begin{equation}
    \label{DiffTheoremTempI}
    \left\| (1+ D_t^2 )^{- \frac 14} (1+ D_0^2)^{\frac 14} - 1 \right\|
    \leq c_0 \, \left\| \, (D_t - D_0)\, (1 + D_0^2)^{- \frac 12}
    \right\|. 
  \end{equation}
  For the second limit in~(\ref{DiffTheoremAimI}), observe that the
  operator~$T_{\psi, t}$ is precisely the operator~$T_{\psi, t}$
  introduced in the proof of~\cite[Theorem~21]{PoSuNGapps}.  Thus, the
  required limit follows from~\cite[Formula~(6.8)]{PoSuNGapps} and the
  estimate (see Lemma~\ref{AuxilliaryEsts}.(\ref{AuxilliaryEstII})
  below)
  \begin{equation}
    \label{DiffTheoremPoSuImprovement}
    \left\| (1 +
      D_t^2)^{\frac {is} 2} - (1 + D_0^2)^{\frac {is} 2} \right\| \leq c_0
    \, \left\| \, (D_t - D_0) \, (1 + D_0)^{- \frac 12} \right\|,\ \
    \left| s \right| \leq s_0, 
  \end{equation}
  for some constant~$c_0 > 0$ which may depend on~$s_0$.  The latter
  estimate is an improvement of~\cite[Formula~(6.16)]{PoSuNGapps}.  The
  proof of the theorem is finished.
\end{proof}

\begin{proof}[Proof of Theorem~\ref{ContDiffTheorem}]
  The proof is similar to the proof of Theorem~\ref{DiffTheorem}.
  Indeed, letting~$T_{\phi, t, s} = T_\phi(D_t, D_s)$ ($\phi$ is given
  in~(\ref{PhiDef})), we obtain
  \begin{align*}
    H_t - H_0 = T_{\phi, t, t} (G_t) - T_{\phi, 0, 0} (G_0) = &\, 
    T_{\phi, t, t} (G_t) - T_{\phi, t, t} (G_0) \cr + &\, T_{\phi,
      t, t} (G_0) - T_{\phi, t, 0} (G_0) \cr + &\, T_{\phi, t, 0} (G_0) -
    T_{\phi, 0, 0} (G_0). 
  \end{align*}
  The first term vanishes since the operators~$T_{\phi, t, t}$ are
  uniformly bounded and the assumption of the theorem; the last one does
  due to~(\ref{DiffTheoremAim}).  To finish the proof, we need to show
  that
  $$ \lim_{t \rightarrow 0} \left\| T_{\phi, t, t} (G_0) - T_{\phi, t,
      0} (G_0) \right\| = 0. $$ Let~$T_{\psi, t, s} = T_\psi (D_t,
  D_s)$, where the function~$\psi$ is given in~(\ref{PsiDef}) and let $$
  C_{t, s} = (1 + D_t^2)^{-\frac 14} G_0 \, (1 + D_s^2)^{\frac 14} \in
  \Bd(\sH). $$ It then follows from the definition of the
  operator~$T_{\phi, t, s}$ that
  \begin{multline}
    \label{ContDiffTheoremTempI}
    T_{\phi, t, t} (G_0) - T_{\phi, t, 0} (G_0) = T_{\psi, t, t} (C_{t, t}) -
    T_{\psi, t, 0} (C_{t, 0}) \\ = T_{\psi_, t, 0} (C_{t, 0} - C_{0, 0})
    + T_{\psi, t, t}(C_{t, t} - C_{0, 0}) \\ + T_{\psi, t, t} (C_{0,
      0}) - T_{\psi, t, 0} (C_{0, 0}).
  \end{multline}
  The first term vanishes due to the fact that the operators~$T_{\psi,
    t, s}$ are uniformly bounded and~(\ref{DiffTheoremAimI}).  For the
  last term in~(\ref{ContDiffTheoremTempI}), observe that the
  operator~$T_{\psi, t, s}$ is the operator~$\bar T_{t,s}$ introduced in
  the proof of~\cite[Theorem~21]{PoSuNGapps}.  Therefore, the last term
  in~(\ref{ContDiffTheoremTempI}) vanishes due
  to~\cite[Formula~(6.11)]{PoSuNGapps} and the
  estimate~(\ref{DiffTheoremPoSuImprovement}).  Thus, to finish the
  proof of the theorem, we need only justify that 
  \begin{equation}
    \label{ContDiffTheoremTempII}
    \lim_{t
      \rightarrow 0} \left\| C_{t, t} - C_{0, 0} \right\| = 0. 
  \end{equation}
  For the latter, observe that
  \begin{align*}
    C_{t, t} - C_{0, 0} = &\, (1+ D_t^2)^{- \frac 14} G_0 \, (1 +
    D_t^2)^{\frac 14} - (1 + D_0^2)^{- \frac 14} G_0 \, ( 1 +
    D_0^2)^{\frac 14} \cr = &\, C_{t, t} \, \left( 1 - (1 + D_t^2)^{-
        \frac 14} ( 1 + D_0^2)^{\frac 14} \right) \cr + &\, \left( (1 +
      D_t^2)^{- \frac 14} (1 + D_0^2)^{\frac 14} - 1 \right) \, C_{0,
      0}.
  \end{align*}
  It is now clear that~(\ref{ContDiffTheoremTempII}) follows
  from~(\ref{DiffTheoremTempI}) and the fact that the operators~$C_{t,
    t}$ are uniformly bounded.  The proof of the theorem is finished.
\end{proof}

\begin{lemma}
  \label{AuxilliaryEsts}
  Let~$D_j :\dom (D_j) \mapsto \sH$, $j = 0,1$ be a linear self adjoint
  operator.  If~$D_1 - D_0$ is $D_0$-bounded and~$\left\| D_1 - D_0
  \right\|_{D_0} \leq \frac 12$, then
  \begin{enumerate}
  \item there is a constant~$c_0 > 0$ such that $$ \left\| (1 +
      D_1^2)^{-\frac 14} (1 + D_0^2)^{\frac 14} - 1 \right\| \leq c_0 \,
    \left\| \, (D_1 - D_0) \, (1 + D_0^2)^{-\frac 12}\right\|;
  $$   \label{AuxilliaryEstI}
\item for every~$s_0 > 0$, there is a constant~$c_1 > 0$ such that $$
  \left\| (1 + D_1^2)^{\frac {is}2} - (1 + D_0^2)^{\frac {is}2} \right\|
  \leq c_1 \, \left\| \, (D_1 - D_0) \, (1 + D_0^2)^{- \frac 12}
  \right\|, \ \ \left| s \right| \leq s_0. $$ \label{AuxilliaryEstII}
\end{enumerate}
\end{lemma}

\begin{proof}
  (\ref{AuxilliaryEstI})~Let us set~$G_j = (1 + D_j^2)^{\frac 12}$, $j =
  0, 1$ for brevity.  By Lemma~\ref{ToSymLemma}, we have $$
  \left\| G_1^{-\frac 34} (D_1 - D_0) \, G_0^{-\frac 14} \right\| \leq
  c_0 \, \left\| \, (D_1 - D_0) \, G_0^{-1} \right\|. $$ Thus, it is
  sufficient to show that $$ \left\| G_1^{-\frac 12} G_0^{\frac 12} - 1
  \right\| \leq c_0\, \left\| G_1^{- \frac 34} (D_1 - D_0)\, G_0^{-\frac
      14} \right\|. $$ Consider the function~
  \begin{equation}
    \label{AuxilliaryPhi}
    \eta (\lambda_1, \lambda_0) = \gamma_1^{\frac 34}
    \frac {\gamma_1^{- \frac 12} \gamma_0^{\frac 12} - 1}{\lambda_1 -
      \lambda_0} \, \gamma_0^{\frac 14},\ \ \gamma_j = (1 +
    \lambda^2_j)^{\frac 12},\ j = 0,1. 
  \end{equation}
  Suppose that~$\eta \in \Phi(\Bd(\sH))$.  The required estimate then
  follows from Theorem~\ref{PerturbationInM} which guarantees the
  identity $$ G_1^{-\frac 12} G_0^{\frac 12} - 1 = T_{\eta} (G_1^{-\frac
    34} (D_1 - D_0) \, G_0^{-\frac 14}), $$ where~$T_\eta = T_\eta(D_1,
  D_0)$.  Thus, to finish the proof of~(\ref{AuxilliaryEstI}) we need
  show that~$\eta \in \Phi(\Bd(\sH))$.  This is justified
  by~\cite[Lemmas~7 and~9]{PoSuNGapps} and the following representation
  of the function~$\eta$ given in~(\ref{AuxilliaryPhi}) $$
  \eta(\lambda_1, \lambda_0) = \frac {\lambda_1}{\gamma_1} \, f\left(
    \frac {\gamma_0} {\gamma_1} \right) + \frac {\lambda_0}{\gamma_0} \,
  f \left( \frac {\gamma_1} {\gamma_0} \right),$$ where the function~$f$
  is given by $$ f(t) = { \left( 1 + t \right)^{-1} \, \left( t^{\frac
        14} + t^{- \frac 14} \right)^{-1}},\ \ t > 0.  $$

  (\ref{AuxilliaryEstII})~We keep the notations of the proof above.
  Let~$s_0$ be fixed.  Referring to Lemma~\ref{ToSymLemma} again, we
  need only show that $$ \left\| G_1^{is} - G_0^{is} \right\| \leq c_0
  \, \left\| G_1^{-\frac 12} (D_1 - D_0) \, G_0^{-\frac 12} \right\|. $$
  Let~$$ \zeta(\lambda_1, \lambda_0) = \gamma_1^{\frac 12} \, \frac
  {\gamma_1^{is} - \gamma_0^{is}}{\lambda_1 - \lambda_0} \,
  \gamma_0^{\frac 12}. $$ If~$\zeta \in \Phi(\Bd(\sH))$, then we have
  the identity (see Theorem~\ref{PerturbationInM}) $$ G_1^{is} -
  G_0^{is} = T_\zeta \left( G_1^{-\frac 12} (D_1 - D_0) \, G_0^{- \frac
      12} \right), $$ where~$T_\zeta = T_\zeta(D_1, D_0)$
  and~(\ref{AuxilliaryEstII}) follows.  Thus, we need to establish
  that~$\zeta \in \Phi(\Bd(\sH))$.  To this end, note that the latter
  function admits the representation $$ \zeta(\lambda_1, \lambda_0) =
  \gamma_1^{\frac {is}2} \gamma_0^{\frac {is}2} \, \left[ \frac
    {\lambda_1}{\gamma_1} \, f\left( \frac {\gamma_0} {\gamma_1} \right)
    + \frac {\lambda_0}{\gamma_0} \, f\left( \frac {\gamma_1} {\gamma_0}
    \right) \right], $$ where the function~$f$ is given by
  $$ f(t) = \frac{t^{\frac {is}2 } - t^{- \frac {is} 2}} {\left( 1 + t
    \right) \, \left( t^{\frac 14} - t^{-\frac 14} \right)}. $$ This,
  together with~\cite[Lemmas~7 and~9]{PoSuNGapps}, implies that~$\zeta
  \in \Phi(\Bd(\sH))$ (with the norm depending on~$s$).  The proof of
  the lemma is finished.
\end{proof}



\section{Appendix}
\label{sec:appendix}

\subsection*{
Homotopical equivalence between~$\Fred_*$ and~$\Fred_*^{\pm 1}$}
We regard the sets~$\Fred_*$ and~$\Fred_*^{\pm 1}$ as topological
spaces endowed with norm topology.

\begin{theorem}
  \label{RetractTheorem}
  The space~$\Fred_*^{\pm 1}$ is a deformation retract of the
  space~$\Fred_*$ i.e., there is a continuous mapping~$r : [0, 1]\times
  \Fred_* \mapsto \Fred_*$ such that
  \begin{enumerate}
  \item $r(0, F) = F$, $F \in \Fred_*$;
  \item $r(1, F) \in \Fred_*^{\pm 1}$, $F \in \Fred_*$;
  \item $r(1, F) = F$, $F \in \Fred_*^{\pm 1}$.
  \end{enumerate}
\end{theorem}

\begin{proof}
 Recall that $\Comp$ is the two-sided ideal of all $\tau$-compact operators
  of~$\aM$ and  
  $\pi$ is the homomorphism $$ \pi : \aM \mapsto \aM / \Comp. $$ 
  Also recall that if~$F$ is a self-adjoint $\tau$-Fredholm operator
  and~$\delta_F = \left\| \pi(F)^{-1} \right\|^{-1}$, then, for every~$0
  \leq \delta < \delta_F$, the spectral projection~$E^{F}(-\delta,
  \delta)$ is $\tau$-finite (see Lemma~\ref{FredholmProperty}) i.e., $$
  \tau \left( E^F(-\delta, \delta) \right) < + \infty,\ \ 0 \leq \delta
  < \delta_F. $$


  Consider the intermediate space~$\Fred_*^{\pm 1} \subseteq \Fred_*'
  \subseteq \Fred_*$ of all $\tau$-Fredholm operators~$F$ such that the
  projection~$E^F(-\delta, \delta)$ is $\tau$-finite, for every~$0\leq
  \delta < 1$.  Clearly, it is sufficient to show that~$\Fred_*'$ is a
  deformation retract of~$\Fred_*$ and that~$\Fred_*^{\pm 1}$ is a
  deformation retract of~$\Fred_*'$.

  Observing that the function~$F \mapsto \left\| \pi(F)^{-1} \right\|$
  is continuous, we see that the deformation retract between~$\Fred_*$
  and~$\Fred_*'$ is given by the mapping $$ r_1(t, F) = F \, \left(1 - t +
    t \, \left\| \pi(F)^{-1} \right\|\right)^{-1} \in \Fred_*,\ \ 0 \leq t
  \leq 1. $$

  To construct the deformation retract between~$\Fred_*'$
  and~$\Fred_*^{\pm 1}$, let us consider the continuous
  function~$\chi(t)$ which is constantly~$1$ for~$t \geq 1$,
  constantly~$-1$ for~$t \leq -1$ and linear for~$-1 \leq t \leq 1$.
  The function~$\chi$ is given by $$ \chi(t) = \frac 12 \left| t + 1
  \right| - \frac 12 \left| t - 1 \right|. $$ Observe that the
  mapping~$F \mapsto \chi(F)$ is continuous in the norm topology
  (see~\cite[Theorem~X.2.1]{Bhatia1997}).  Observe also that~$\chi(F)
  \in \Fred_*^{\pm 1}$, for every~$F \in \Fred_*'$.  Indeed, let us
  fix~$F \in \Fred_*'$ and let~$F_1 = \chi(F)$.  Clearly,~$\left\| F_1
  \right\| \leq 1$.  To see that~$1 - F_1^2$ is $\tau$-compact, consider the
  points~$\delta_n = 1 - \frac 1n$, $n = 1, 2, \ldots\,$, and the
  projections $$ E_n = E^F (-\delta_{n+1}, - \delta_n] + E^F [\delta_n,
  \delta_{n+1}). $$ Observe that, since~$F \in \Fred_*'$, the
  projection~$E_n$ is $\tau$-finite, for every~$n = 1, 2, \ldots\,$.
  Noting that $$ E^F (-1, 1) = E^{F_1} (-1, 1) \ \ \text{and}\ \ (1 -
  F_1^2) E^{F_1} \left\{ \pm 1 \right\} = 0, $$ we obtain that
  \begin{equation*}
    1 - F_1^2 = (1 - F_1^2) \, E^F (-1, 1) = (1 - F_1^2) \, \sum_{n =
      1}^\infty E_n = \sum_{n = 1}^\infty (1 - F_1^2) \, E_n \leq
    \sum_{n = 1}^\infty \frac 2n \, E_n.
  \end{equation*}
  The latter means that the operator~$1 - F_1^2$ is $\tau$-compact.
  Thus, we may define the deformation retract between~$\Fred_*'$
  and~$\Fred_*^{\pm 1}$ by setting $$ r_2(t, F) = (1 - t)\, F + t\,
  \chi(F) \in \Fred_*',\ \ 0 \leq t \leq 1. $$ The theorem is proved.
\end{proof}

\def\cprime{$'$}

\end{document}